\documentclass[12pt, a4paper]{amsart}
\usepackage{amsmath}
\usepackage{geometry,amsthm,graphics,tabularx,amssymb,shapepar}
\usepackage{amscd}
\usepackage[all]{xypic}
\newtheorem{thm}{Theorem}[section]
\newtheorem{lem}[thm]{Lemma}
\newtheorem{cor}[thm]{Corollary}
\newtheorem{prop}[thm]{Proposition}
\theoremstyle{definition}

\theoremstyle{remark}
\newtheorem{rem}[thm]{Remark}
\numberwithin{equation}{section}

\newcommand{\la}{{\langle}}
\newcommand{\ra}{{\rangle}}
\makeatletter \@addtoreset{equation}{section}

\makeatother \makeatletter

\begin{document}

\begin{center}

{\Large \bf Twisted $\Gamma$-Lie
algebras and their vertex operator representations}

\end{center}

\begin{center}
{Fulin Chen$^{a}$, Shaobin Tan$^{b}$\footnote{Partially supported by
NSF of China (No.10931006) and a grant from the PhD Programs
Foundation of Ministry of Education of China (No.20100121110014).}
and Qing Wang$^{b}$\footnote{Partially supported by NSF of China (No.11371024),
Natural Science Foundation of Fujian Province (No.2013J01018) and
Fundamental Research Funds for the Central
University (No.2013121001).
}\\
$\mbox{}^{a}$Academy of Mathematics and System Science\\ Chinese
Academy of Sciences, Beijing 100190, China\\
$\mbox{}^{b}$School of Mathematical Sciences\\ Xiamen University,
Xiamen 361005, China}
\end{center}

\vspace{6mm}





\textbf{Abstract.}  Let $\Gamma$ be a generic subgroup of the
multiplicative group $\mathbb{C}^*$ of nonzero complex numbers. We
define a class of Lie algebras associated to $\Gamma$, called
twisted $\Gamma$-Lie algebras, which is a natural generalization of
the twisted affine Lie algebras. Starting from an arbitrary even sublattice $Q$ of $\mathbb Z^N$ and
an arbitrary finite order isometry of $\mathbb Z^N$ preserving $Q$, we construct a family of twisted
$\Gamma$-vertex operators acting on generalized Fock spaces
which afford irreducible representations for certain twisted
$\Gamma$-Lie algebras. As application, this recovers a number of
known vertex operator realizations for infinite dimensional Lie
algebras, such as
 twisted affine Lie algebras, extended affine Lie algebras of type
$A$, trigonometric Lie algebras of series $A$ and $B$, unitary Lie
algebras, and $BC$-graded Lie algebras.

\section{Introduction}\label{sec 1}
Affine Kac-Moody Lie algebras, which are a family of
infinite-dimensional Lie algebras, have played an important role in
mathematics and mathematical physics. One striking discovery in the
representation theory of affine Kac-Moody algebras was the vertex
operator construction of the basic representations. The first such
construction, called principal, was discovered in \cite{LW} and
generalized in \cite{KKLW}. Another construction, called
homogeneous, was given in \cite{FK} and \cite{S} independently.
 Besides affine Kac-Moody algebras, there are some
other interesting infinite dimensional Lie algebras such as extended
affine Lie algebras, trigonometric Lie algebras, unitary Lie algebras and root graded Lie
algebras. And the vertex operator
representations also play an important role in the study of
representation theory of these algebras. To explain our motivation,
we first give a brief introduction to these algebras and mention
some of their vertex operator constructions which are closely
related to our work.

\begin{itemize}
\item \textbf{Twisted affine Lie algebras} The vertex operator representations
for twisted affine Lie algebras were obtained in \cite{L} and
\cite{KP} in a very general setting. In \cite{L}, the author gave a
generalization of all the known vertex operator constructions for
affine Kac-Moody algebras by introducing the twisted vertex
operators associated with an arbitrary isometry of an arbitrary even
lattice.

\item \textbf{Extended affine Lie algebras of type $A$.}
Extended affine Lie algebras were first introduced in \cite{H-KT}
and systematically studied in \cite{AABGP}. The
 vertex operator representations for extended affine Lie
algebras of type $A$ coordinated by quantum tori have been given in
\cite{BS} and \cite{G2} for the principal realizations, and in
\cite{G1} in the homogeneous realization. Later in \cite{BGT}, a
unified treatment was given.

\item \textbf{Trigonometric Lie algebras.} As a natural generalization of the Sine Lie algebra, four
series of $\mathbb Z$-graded trigonometric subalgebras
$\widehat{A}_\mathbf h-\widehat{D}_\mathbf h$ of
$\widehat{A}_\infty-\widehat{D}_\infty$ were introduced in
\cite{G-KL1,G-KL2}. Moreover, the vertex operator constructions for
the Lie algebras $\widehat{A}_\mathbf h$ and $\widehat{B}_\mathbf h$
were also presented in \cite{G-KL1,G-KL2}.

\item \textbf{Unitary Lie algebras.}  Unitary Lie algebras
were first introduced in \cite{AF}. The compact forms of certain
intersection matrix algebras developed by Slodowy can be identified
with some Steinberg unitary Lie algebras. The vertex
operator constructions for a class of unitary Lie algebras
coordinated by skew Laurent polynomial rings were presented in
\cite{CGJT}.

\item \textbf{$BC$-graded Lie algebras.}
Root graded Lie algebras were first introduced in \cite{BM}
 and the $BC$-graded Lie algebras were studied and classified in \cite{ABG}.
The homogeneous vertex operator constructions for a class of
$BC$-graded Lie algebras coordinated by
 skew Laurent polynomial rings were given in \cite{CT}.

\end{itemize}

Though these Lie algebras were introduced for different purposes and
were defined by different approaches, we observe that their vertex
operator constructions arise from certain known vertex operators for
affine Lie algebras of type $A$ or $D$ and depend on certain nonzero
complex numbers. It is natural for us to give a general construction
so that it contains these vertex operator constructions as special
cases. This is the main motivation of our work.

We first give the definition of twisted $\Gamma$-Lie algebras
which simultaneously generalize those Lie algebras mentioned above.
More precisely, let $\Gamma$ be a subgroup of the
multiplicative group $\mathbb C^*$ satisfying the following
assumption:

\vspace{2mm}

(A1). $\Gamma$ is generic, i.e., $\Gamma$ is isomorphic to a free
abelian group.
\vspace{2mm}

We start from any finite order automorphism of an involutive associative algebra to construct the twisted $\Gamma$-Lie algebra.
In particular, we use a lattice to
construct an involutive associative algebra, then a twisted
$\Gamma$-Lie algebra is defined based on this involutive associative
algebra. To explain this process more precisely, we let $N$ be a
positive integer, and $P=\mathbb Z^N=\mathbb Z\epsilon_1\oplus
\cdots \oplus \mathbb Z\epsilon_N$ be a lattice of rank $N$ equipped
with a $\mathbb{Z}$-bilinear form given by
$\la\epsilon_i,\epsilon_j\ra=\delta_{ij}, 1\le i,j\le N$. We take a triple $(Q,\nu,m)$
 satisfying the following conditions:

\vspace{2mm}
(A2). $Q$ is a sublattice of $P$ such that $\la\alpha,\alpha\ra\in
2\mathbb{Z}$ for $\alpha\in Q$.

(A3). $\nu$ is an isometry of $P$ preserving $Q$. That is, $\nu$ is
an automorphism of $P$ such that
$\la\nu\alpha,\nu\beta\ra=\la\alpha,\beta\ra$ for $\alpha,\beta\in
P$ and
 $\nu(Q)=Q$.

(A4). $m$ is a positive integer such that $\nu^m=\mathrm{Id}$, the identity map on
$P$.

(A5). If $m$ is even, $\la\nu^{m/2}\alpha,\alpha\ra\in 2\mathbb Z$
for $\alpha\in Q$.
\vspace{2mm}

Note that if the triple $(Q,\nu,m)$ satisfies the assumptions
(A2)-(A4), then the assumption (A5) can always be arranged by
doubling $m$ if necessary. Now we give some examples of the triples
$(Q,\nu,m)$ which satisfy assumptions (A2)-(A5). Let
$Q=Q(A_{N-1}),N\ge 2$ or $Q(D_N),N\ge 5$ be the root lattices of type $A_{N-1}$ and $D_N$
respectively,
 and let $\nu$ be an isometry of $Q$. It is known that $\nu$ can be lifted to be an isometry of $P$ and has finite order.
Thus there exists a positive integer $m$ such that $(Q,\nu,m)$
satisfies assumptions (A2)-(A5).

Starting from a quadruple $(Q,\nu,m,\Gamma)$ which satisfies
(A1)-(A5), we construct a twisted $\Gamma$-Lie algebra $\widehat{\mathcal G}(Q,\nu,m,\Gamma)$
 and  define a family of twisted $\Gamma$-vertex
operators acting on a generalized Fock space. By computing the
commutator relations of these twisted $\Gamma$-vertex operators, we
obtain a class of irreducible representations for the twisted
$\Gamma$-Lie algebra $\widehat{\mathcal G}(Q,\nu,m,\Gamma)$. One
will see that it is very subtle and technical to determine the
commutator relations for the twisted $\Gamma$-vertex operators. To
do this, we develop a highly non-trivial generalization of the
combinatorial identity presented in Proposition 4.1 of \cite{L}. As
applications, we show that with different choices of quadruples
$(Q,\nu,m,\Gamma)$, we recover the vertex operator constructions of the
twisted affine Lie algebras, extended affine Lie algebras of type
$A$ (both homogeneous and principal constructions), trigonometric
Lie algebras of series $A$ and $B$, unitary Lie algebras
 and $BC$-graded
Lie algebras given in \cite{L,G-KL1,G-KL2,BS,G1,G2,BGT,CGJT,CT}
respectively. Moreover, we also present  vertex operator representations
for a new twisted $\Gamma$-Lie algebras.


The paper is organized as follows. In Section 2 we give the
definition of the twisted $\Gamma$-Lie algebra $\widehat{\mathcal
G}(Q,\nu,m,\Gamma)$ for any
quadruple $(Q,\nu,m,\Gamma)$ satisfying the assumptions (A1)-(A5).
 In Section 3 we define
 the generalized Fock space and give the twisted $\Gamma$-vertex operators.
In Section 4 we establish a crucial identity (see \eqref{eq:3.15}) by using certain combinatorial identities, which
allows us to compute the commutator relations of the twisted
$\Gamma$-vertex operators. In Section 5 we prove our main
result in Theorem \ref{thm:4.1}, which shows that those twisted
$\Gamma$-vertex operators acting on generalized Fock spaces give
representations for the Lie algebra $\widehat{\mathcal
G}(Q,\nu,m,\Gamma)$. Finally, in Section 6 we present the
applications of our main result.

\textbf{Conventions and notations:}

1. We denote the sets of integers, positive integers,
 complex numbers, real numbers and the quotient group $\mathbb Z/m\mathbb Z$ by
$\mathbb Z, \mathbb Z_+, \mathbb C, \mathbb R$ and $\mathbb Z_m$,
respectively.

2. For short, if $\alpha,\beta\in P$, we often write
$$\sum_{p\in \mathbb Z_m}\nu^p\alpha:=\sum \nu^p\alpha,\text{ and }
\sum_{p\in \mathbb Z_m}\la \alpha,\nu^p\beta\ra:=\sum
\la\alpha,\nu^p\beta\ra.$$

3. For $1\le i\ne j\le N$ and $c\in \mathbb
C$, we set $(1-c)^{\pm\delta_{ij}}=1$ even if $c=1$.

4.  Since the fraction powers of nonzero complex numbers will arise
in the construction of our vertex operators, we need a convention
for this. For an index set $I$, we fix a set of free generators
$\{q_i,i\in I\}$ of $\Gamma$, and fix a choice of $q_i^{\frac 1 2}$ for each $i\in I$.  For any
$c=q_{i_1}^{n_1}\cdots q_{i_t}^{n_t}\in \Gamma$, where
$i_1,\cdots,i_t\in I$, $n_1, \cdots,n_t\in \mathbb{Z}$, we set
$$c^{\frac 1 2}:=(q_{i_1}^{\frac 1 2})^{n_1}\cdots (q_{i_t}^{\frac 1 2})^{n_t}.$$
Then we have $c_1^{\frac n 2}c_2^{\frac n 2}=(c_1c_2)^{\frac n 2}$
for all $c_1,c_2\in \Gamma$ and $n\in \mathbb Z$.


5. Set $Q(D_1)=2\mathbb Z\epsilon_1$, and for $N\geq 2,$ let
\begin{align*}Q(A_{N-1})&=\mathbb Z(\epsilon_1-\epsilon_2)\oplus \cdots \oplus \mathbb Z(\epsilon_{N-1}-\epsilon_N),\\
Q(D_{N})&=\mathbb Z(\epsilon_1-\epsilon_2)\oplus \cdots \oplus
\mathbb Z(\epsilon_{N-1}-\epsilon_N) \oplus \mathbb
Z(\epsilon_{N-1}+\epsilon_N)\end{align*} be the root lattices of
type $A_{N-1}$ and $D_N$, respectively.

\section{Twisted $\Gamma$-Lie algebras}
In this section, we give the definition of the twisted
$\Gamma$-Lie algebra and present some examples.
In particular, a twisted $\Gamma$-Lie
algebra $\widehat{\mathcal G}(Q,\nu,m,\Gamma)$ associated to any
quadruple $(Q,\nu,m,\Gamma)$ satisfying (A1)-(A5) is constructed.

\subsection{Twisted $\Gamma$-Lie algebras}
Let $\mathbb C_\Gamma$ be an associative algebra with base elements of the form $t^nT_c, n\in \mathbb Z,c\in
\Gamma$, and multiplication given by
$$(t^nT_{c_1})(t^rT_{c_2})=c_1^rt^{n+r}T_{c_1c_2},\ n,r\in \mathbb Z, c_1,c_2\in \Gamma.$$

Let $\mathcal A$ be another associative algebra equipped with an invariant symmetric bilinear form
 $\la , \ra_\mathcal A$. Suppose that $\theta$ is a finite order automorphism of $\mathcal A$ and
 preserves the bilinear form $\la , \ra_\mathcal A$.
 Let $m$ be a positive integer such that $\theta^m=\mathrm{Id}$ and fix a primitive $m$-th
 root  of unity $\omega$. Viewing $\mathcal A\otimes \mathbb C_\Gamma$ as a Lie algebra, denoted by $\mathcal A(\Gamma)$, with Lie product
$[a\otimes b,a'\otimes b']=aa'\otimes bb'-a'a\otimes b'b$ for
$a,a'\in \mathcal A$ and $b,b'\in \mathbb C_\Gamma$. Extend the
automorphism $\theta$ of $\mathcal A$ to be a Lie automorphism of
$\mathcal A(\Gamma)$, still denoted by $\theta$, by letting
\begin{align*} \theta(a\otimes t^nT_c):=\omega^{-n}\theta(a)\otimes t^nT_c,\ a\in \mathcal A,n\in \mathbb Z, c\in \Gamma.
\end{align*}
 We denote by $\mathcal A(\theta,m,\Gamma)$ the subalgebra of
$\mathcal A(\Gamma)$ fixed by the automorphism
$\theta$.
For any $a\in \mathcal A$ and $n\in \mathbb Z$, set
\begin{align*}
a_{(n)}^\theta=a_{(n)}:=m^{-1}\sum_{p\in \mathbb Z_m}\omega^{-np}\theta^p(a),\
\mathcal A_{(n)}:=\{a_{(n)}|a\in \mathcal A\}.
\end{align*}
Then we have $\theta(a_{(n)})=\omega^n a_{(n)}$ and $\mathcal A=\oplus_{p\in \mathbb Z_m}\mathcal A_{(p)}$.
Obviously,
 the Lie algebra $\mathcal A(\theta,m,\Gamma)$ is spanned by the following elements
 \begin{align*} a(c,n):=a_{(n)}\otimes t^nT_c,\ a\in \mathcal A, c\in \Gamma, n\in\mathbb Z.
 \end{align*}

  We define a bilinear form $\la , \ra$ on the Lie algebra $\mathcal A(\theta,m,\Gamma)$ as follows
 $$\la a(c_1,n),b(c_2,r)\ra:= \frac{n}{2m}\delta_{n+r,0}\la a,b\ra_\mathcal A \delta_{c_1c_2,1}c_1^r,\ a,b\in \mathcal A,
 c_1,c_2\in \Gamma, n,r\in \mathbb Z.
 $$
It is easy to see that the  bilinear form $\la , \ra$ is a
2-cocycle on the Lie algebra $\mathcal A(\theta,m,\Gamma)$. We
define a Lie algebra $\widehat{\mathcal
A}(\theta,m,\Gamma):=\mathcal A(\theta,m,\Gamma)\oplus
\mathbb{C}\mathbf{c}$
 to be the 1-dimensional
central extension  of $\mathcal A(\theta,m,\Gamma)$ associated to this 2-cocycle.

Let $\tau$ be an anti-involution of $\mathcal A$ which preserves
$\la ,\ra_\mathcal A$. We simply call the pair $(\mathcal A,\tau)$ an involutive associative algebra.
Suppose that $\theta$ is also an automorphism
of the involutive associative algebra $(\mathcal A,\tau)$, i.e., it commutes with $\tau$. Define an
anti-involution  $\ \bar{ }\ $ on $\mathbb C_\Gamma$  by
$$\overline{t^nT_c}:=c^{-n}t^nT_{c^{-1}},\ n\in \mathbb Z, c\in \Gamma.$$
Consider the linear map $\hat{\tau}$ on the Lie algebra $\widehat{\mathcal A}(\theta,m,\Gamma)$ given by
\begin{align*}
\hat{\tau}(a_{(n)}\otimes t^nT_c):=-\tau(a)_{(n)}\otimes \overline{t^nT_c}, \
\hat{\tau}(\mathbf c)=\mathbf c,\ a\in \mathcal A,n\in\mathbb Z, c\in \Gamma,
\end{align*}
which is a Lie involution of $\widehat{\mathcal A}(\theta,m,\Gamma)$.
 Now we define the twisted $\Gamma$-Lie algebra $\widehat{\mathcal A}_\tau(\theta,m,\Gamma)$ to be
 the set of fixed-points of $\widehat{\mathcal A}(\theta,m,\Gamma)$ under the involution $\hat{\tau}$.
Then one can see that the following elements
 \begin{align}\label{eq:e1}
 \tilde{a}(c,n):=a(c,n)+\hat{\tau}(a(c,n)),\ a\in \mathcal A, c\in \Gamma, n\in\mathbb Z,
 \end{align}
 together with the central element $\mathbf c$, span the algebra $\widehat{\mathcal A}_\tau(\theta,m,\Gamma)$.

\begin{rem} In the Lie algebra $\widehat{\mathcal A}_\tau(\mathrm{Id},1,\Gamma)$, we define a formal power
series $\tilde{a}(c,z)=\sum_{n\in\mathbb Z}\tilde{a}(c,n)z^{-n}$ for $a\in \mathcal A$ and $c\in \Gamma$.
Note that the formal power series satisfying the following so-called ``$\Gamma$-locality" (\cite{G-KK},\cite{Li})
$$(z_1-z_2)(c_1z_1-z_2)(z_1-c_2z_2)(c_1z_1-c_2z_2)[\widetilde{a_1}(c_1,z_1),\widetilde{a_2}(c_2,z_2)]=0,$$
where $a_1,a_2\in \mathcal A$ and $c_1,c_2\in \Gamma$. Motivated by the notion of $\Gamma$-conformal
algebra defined in \cite{G-KK} and $\Gamma$-vertex algebra defined in \cite{Li}, we call the Lie algebra
$\widehat{\mathcal A}_\tau(\theta,m,\Gamma)$ the twisted $\Gamma$-Lie algebra.
\end{rem}

 \begin{rem}Let $\mathcal A^{\mathrm{op}}$ be the opposite algebra of $\mathcal A$ and $\mathrm{ex}$ be the
exchange involution of $\mathcal A\oplus \mathcal A^{\mathrm{op}}$,
i.e., $\mathrm{ex}(a,a^{'})=(a^{'},a)$.
 Extend the automorphism $\theta$ and the bilinear
form $\la , \ra_\mathcal A$ of $\mathcal A$   to be an automorphism
and a bilinear form of $(\mathcal A\oplus \mathcal
A^{\mathrm{op}},\mathrm{ex})$ by
$\theta(a,a'):=(\theta(a),\theta(a'))$ and $\la (a,a'),(b,b')\ra_\mathcal A:=\la a,b\ra_\mathcal A+\la a',b'\ra_\mathcal A$,
respectively. Then we have
$\widehat{\mathcal A}(\theta,m,\Gamma)\cong \widehat{\mathcal A\oplus \mathcal A^{\mathrm{op}}
}_\mathrm{ex}(\theta,m,\Gamma).$
\end{rem}

\subsection{Examples} We denote by $\mathcal M_N$ the $N\times N$-matrix algebra over $\mathbb{C}$, and denote by $E_{i,j},1\le i,j\le N$
the unit matrices in $\mathcal M_N$. In what follows,
if $\mathcal A=\mathcal M_N$,
then the invariant bilinear form $\la ,\ra_\mathcal A$ is always taken to be the trace form.

\textbf{1. Twisted affine Lie algebras.} Let $(\mathcal A,\tau)$  be
a finite dimensional simple involutive associative algebra and
$\Gamma=\{1\}$,
 then the Lie algebra $\widehat{\mathcal A}_\tau(\theta,m,\Gamma)$ is a twisted affine Lie algebra. Moreover,
 all the classical affine Lie algebras, that is, the affine Lie algebras of type $X_l^{(r)}, r=1,2$ and
 $X=A,B,C,D$, can be realized by this way.

\textbf{2. Extended affine Lie algebras
$\widehat{\mathfrak{gl}}_N(\mathbb C_q)$.} Recall
 the Lie algebra $\widehat{\mathfrak{gl}}_N(\mathbb
C_q)$ defined in \cite{BGT}, which is a 1-dimensional central extension of the matrix algebra over the quantum torus $\mathbb
C_q$ associated to $q=(q_1,\cdots, q_l)\in (\mathbb C^{*})^l.$  Explicitly, it has a basis
 $E_{i,j}t_0^{n_0}t^\mathbf n$ and $ \mathbf c$, for $ 1\le i,j\le N,
\mathbf n\in \mathbb Z^l, n_0\in \mathbb Z,$  subject to
the Lie  relation
\begin{equation*}
\begin{split}
[E_{i,j}t_0^{n}t^\mathbf n,E_{k,p}t^{r_0}t^\mathbf r]=& q^{r_0\mathbf
n}\delta_{jk}E_{i,p}t_0^{n_0+r_0}t^{\mathbf n+\mathbf r}-
q^{n_0\mathbf r}\delta_{ip}E_{k,j}t_0^{n_0+r_0}t^{\mathbf n+\mathbf r}\\
&+n_0q^{n_0\mathbf
r}\delta_{jk}\delta_{ip}\delta_{n_0+r_0,0}\delta_{\mathbf n+\mathbf
r,0}\mathbf c,
\end{split}
\end{equation*}
where $t^\mathbf n=t_1^{n_1}\cdots t_l^{n_l}$ and $q^\mathbf
n=q_1^{n_1}\cdots q_l^{n_l}$, and $\mathbf c$ is central.
Assume that the subgroup $\Gamma_q$
of $\mathbb C^*$ generated by $q_1,\cdots,q_l$ is a free abelian group of rank $l$.
Then the Lie algebra $\widehat{\mathcal M_N\oplus \mathcal M_N^{\mathrm{op}}}_{\mathrm{ex}}(\mathrm{Id},1,\Gamma_q)$ is isomorphic to $\widehat{\mathfrak{gl}}_N(\mathbb C_q)$
via the isomorphism
$$\widetilde{E_{i,j}}(q^{\mathbf n},n_0)\mapsto E_{i,j}t_0^{n}t^\mathbf n,\ \mathbf c\mapsto \mathbf c,\  1\le i,j\le N, (n_0,\mathbf n)\in \mathbb Z^{l+1}.$$

\textbf{3. Trigonometric Lie algebras} (cf.\cite{G-KL1,G-KL2}). Given
a $l$-tuple $\mathbf h=(h_1,\cdots,h_l)\in \mathbb R^l$, we set
$\Gamma_\mathbf h=\{e^{2\sqrt{-1}(\mathbf h,\mathbf n)}|\mathbf
n=(n_1,\cdots,n_l)\in\mathbb Z^l\}$, where $(\mathbf h,\mathbf
n)=h_1n_1+\cdots+h_ln_l$.
 Assume that $h_1,\cdots, h_l$ are $\mathbb Z$-linearly independent so that $\Gamma_\mathbf h$ is generic.
Set $\mathcal A=\mathbb C\oplus \mathbb C,\tau=\mathrm{ex}$ and
$\Gamma=\Gamma_\mathbf h$. Then in $\widehat{\mathcal
A}_\tau(\mathrm{Id},1,\Gamma_\mathbf h)$, one has
\begin{align*} [A_{\mathbf n,n_0}, A_{\mathbf r,r_0}]
=2\sqrt{-1}\sin(n_0(\mathbf h,\mathbf r)-r_0(\mathbf h,\mathbf
n))A_{\mathbf n+\mathbf r,n_0+r_0}
+n_0\delta_{n_0+r_0,0}\delta_{\mathbf n,\mathbf r}\mathbf c,
\end{align*}
where $A_{\mathbf n,n_0}:=e^{-n_0\sqrt{-1}(\mathbf h,\mathbf
n)}\widetilde{(1,0)}(e^{-2\sqrt{-1}(\mathbf h,\mathbf n)},n_0).$
These are the commutator relations of the trigonometric Lie algebra
of series $\widehat{A}_\mathbf h$ defined in \cite{G-KL1}.
Furthermore, in $\widehat{\mathcal A}_\tau(\tau,2,\Gamma_\mathbf
h)$, one has
\begin{equation*}\begin{split}
[B_{\mathbf n,n_0}, B_{\mathbf r,r_0}] =&_{}2\sqrt{-1}\sin(n_0(\mathbf
h,\mathbf r)-r_0(\mathbf h,\mathbf n))B_{\mathbf n+\mathbf
r,n_0+r_0}
+(-1)^{r_0}2\sqrt{-1}\sin(n_0(\mathbf h,\mathbf r)\\
&+r_0(\mathbf h,\mathbf
n))B_{\mathbf n-\mathbf r,n_0+r_0}
+n_0\delta_{n_0+r_0,0}(\delta_{\mathbf n,\mathbf r}-(-1)^{n_0}
\delta_{\mathbf n,-\mathbf r})\mathbf c,
\end{split}\end{equation*}
where $B_{\mathbf n,n_0}:=2e^{-n_0\sqrt{-1}(\mathbf h,\mathbf
n)}\widetilde{(1,0)}(e^{-2\sqrt{-1}(\mathbf h,\mathbf n)},n_0).$
 Then the Lie algebra $\widehat{\mathcal A}_\tau(\tau,2,\Gamma)$
  is isomorphic to the trigonometric Lie algebra of series $\widehat{B}_\mathbf h$ (cf.\cite{G-KL1}).

\textbf{4. Unitary Lie algebras $\widehat{\mathfrak{u}}_N(\mathbb
C_\Gamma)$.}  The Lie algebra $\widehat{\mathfrak{u}}_N(\mathbb
C_\Gamma)$ defined in \cite{CGJT} is spanned by the elements
$u_{i,j}(c,n),\mathbf c, 1\le i,j\le N,c\in \Gamma,n\in \mathbb Z$
with the relation $u_{i,j}(c,n)= -(-c)^{-n}u_{j,i}(c^{-1},n)$, and
subject to
\begin{equation*}\begin{split}
&[u_{i,j}(c_1,n),u_{k,l}(c_2,r)]
=\delta_{jk}c_1^ru_{i,l}(c_1c_2,n+r)\\&+\delta_{il}(-c_1)^{-n-r}c_2^{-r}u_{j,k}(c_1^{-1}c_2^{-1},n+r)
-\delta_{ik}(-1)^nc_1^{-n-r}u_{j,l}(c_1^{-1}c_2,n+r)\\
&-\delta_{jl}c_1^r(-c_2)^{-n}u_{i,k}(c_1c_2^{-1},n+r)+n\delta_{n+r,0}(\delta_{jk}\delta_{il}\delta_{c_1c_2,1}n
-\delta_{ik}\delta_{jl}\delta_{c_1,c_2}(-1)^n)\mathbf c,
\end{split}\end{equation*}
where $\mathbf c$ is a central element.
Choose $\mathcal A=\mathcal M_N\oplus \mathcal M_N^{\mathrm{op}},
\tau=\mathrm{ex},$ $\theta:(A,B)\mapsto (B^t,A^t)$ and $m=2$, where
$A,B\in \mathcal M_N, A^t$ is the transpose of $A$. It is easy to
check that the following linear map gives an isomorphism from the
unitary Lie algebra $\widehat{\mathfrak{u}}_N(\mathbb C_\Gamma)$ to
the twisted $\Gamma$-Lie algebra $\widehat{\mathcal
A}_\tau(\theta,2,\Gamma)$:
$$u_{i,j}(c,n)\mapsto 2\widetilde{(E_{i,j},0)}(c,n),\ \mathbf c\mapsto \mathbf c,\ 1\le i,j\le N,c\in \Gamma, n\in \mathbb Z.$$

\textbf{5. $BC_N$-graded Lie algebras
$\widehat{\mathfrak{o}}_{2N}(\mathbb C_\Gamma)$.} Following
\cite{CT}, the Lie algebra $\widehat{\mathfrak{o}}_{2N}(\mathbb
C_\Gamma)$ is spanned by the elements
$f_{\rho_ii,\rho_jj}(c,n),\mathbf c,\ 1\le i,j \le N,
\rho_i,\rho_j=\pm 1, c\in \Gamma,n\in \mathbb Z,$ with relation
$f_{\rho_ii,\rho_jj}(c,n)=
-(-c)^{-n}f_{-\rho_jj,-\rho_ii}(c^{-1},n)$. The Lie bracket in
$\widehat{\mathfrak{o}}_{2N}(\mathbb C_\Gamma)$ is given by
\begin{equation*}\begin{split}
&[f_{\rho_ii,\rho_jj}(c_1,n),f_{\rho_kk,\rho_ll}(c_2,r)]
=
\delta_{\rho_jj,\rho_kk}
c_1^rf_{\rho_ii,\rho_ll}(c_1c_2,n+r)
+\delta_{\rho_ii,\rho_ll}c_1^{-n-r}c_2^{-r}\\&f_{-\rho_jj,-\rho_kk}(c_1^{-1}c_2^{-1},n+r)
-\delta_{\rho_ii,-\rho_kk}
c_1^{-n-r}
 f_{-\rho_jj,\rho_ll}(c_1^{-1}c_2,n+r)
-\delta_{-\rho_jj,\rho_ll}
(c_1/c_2)^r\\&f_{\rho_ii,-\rho_kk}(c_1c_2^{-1},n+r)+\delta_{\rho_jj,\rho_kk}\delta_{\rho_ii,\rho_ll}\delta_{n+r,0}
c_1^r\delta_{c_1c_2,1}n\textbf{c}
-\delta_{\rho_ii,-\rho_kk}\delta_{\rho_jj,-\rho_ll}
\delta_{n+r,0}\delta_{c_1,c_2}n\mathbf{c}.
\end{split}\end{equation*}
Choose $\mathcal A=\mathcal M_{2N},
\tau:E_{\rho_ii,\rho_jj}\mapsto E_{-\rho_jj,-\rho_ii},$ $\theta=\mathrm{Id}$ and $m=1$, where
$1\le i,j\le N,\rho_i,\rho_j=\pm 1$. Then the following linear map gives an isomorphism from the
$BC_N$-graded Lie algebra $\widehat{\mathfrak{o}}_{2N}(\mathbb C_\Gamma)$ to
the twisted $\Gamma$-Lie algebra $\widehat{\mathcal
A}_\tau(\mathrm{Id},1,\Gamma)$:
$$f_{\rho_ii,\rho_jj}(c,n)\mapsto \widetilde{E_{\rho_ii,\rho_jj}}(c,n),\ \mathbf c\mapsto \mathbf c,\
1\le i,j\le N,\rho_i,\rho_j=\pm 1,c\in \Gamma, n\in \mathbb Z.$$

\subsection{Lattice construction of involutive associative algebras}
One of the most important step in the vertex operator representation
theory of affine Lie algebras is the lattice construction for semi-simple Lie algebras. In this section, we give the
lattice construction for involutive associative algebras, which will
be used to present vertex operator representations for the twisted
$\Gamma$-Lie algebras.

Given a triple $(Q,\nu,m)$
which satisfies assumptions (A2)-(A5). We define a set
\begin{align}\label{eq:1.1}
\mathcal{J}=\{(\rho_ii,\rho_jj)|\rho_i\epsilon_i-\rho_j\epsilon_j\in
Q,1\le i,j\le N, \rho_i,\rho_j=\pm 1\}.
\end{align}
Let
$\mathcal G(Q)=\bigoplus_{(\rho_ii,\rho_jj)\in \mathcal{J}} \mathbb{C}e_{\rho_ii,\rho_jj}$
 be a vector space
over $\mathbb{C}$, where
$\{e_{\rho_ii,\rho_jj}\}_{(\rho_ii,\rho_jj)\in \mathcal J}$
is a set of symbols. We define  multiplication on $\mathcal G(Q)$ by
\begin{align*}
 e_{\rho_ii,\rho_jj}e_{\rho_kk,\rho_ll}=\delta_{\rho_jj,\rho_kk}
 \varepsilon(\rho_i\epsilon_i-\rho_j\epsilon_j,\rho_k\epsilon_k-\rho_l\epsilon_l)
e_{\rho_ii,\rho_ll},
\end{align*}
where $(\rho_ii,\rho_jj),(\rho_kk,\rho_ll)\in \mathcal{J}$ and  $\varepsilon:Q \times Q\rightarrow \mathbb C$ is a
normalized 2-cocycle on $Q$ associated with the function $(-1)^{\la \alpha,\beta\ra}$, that is, it satisfies the conditions
\begin{equation}\begin{split} \label{eq:c1}
   &\varepsilon(\alpha,\beta)\varepsilon(\alpha+\beta,\gamma)=
   \varepsilon(\beta,\gamma)\varepsilon(\alpha,\beta+\gamma),\\
   &\varepsilon(0,0)=1,\ \varepsilon(\alpha,\beta)/\varepsilon(\beta,\alpha)=(-1)^{\la\alpha,\beta\ra}.
\end{split} \end{equation}
 Define a linear map $\tau$ and a bilinear form $\la , \ra_\mathcal G$ on $\mathcal G(Q)$ as follows
 \begin{align*}
\tau(e_{\rho_ii,\rho_jj})&=(-1)^{1-\delta_{ij}} e_{-\rho_jj,-\rho_ii},\\
 \la e_{\rho_ii,\rho_jj},e_{\rho_kk,\rho_ll}\ra_\mathcal G&=\delta_{\rho_jj,\rho_kk}\delta_{\rho_ii,\rho_ll}
\varepsilon(\rho_i\epsilon_i-\rho_j\epsilon_j,\rho_k\epsilon_k-\rho_l\epsilon_l)
\end{align*}
for $(\rho_ii,\rho_jj),(\rho_kk,\rho_ll)\in \mathcal{J}$. It is easy
to see that  $(\mathcal G(Q),\tau)$ is an involutive associative
algebra and $\la , \ra_\mathcal G$ is an invariant symmetric
bilinear form on it.

Recall that $\nu$ is an isometry of $P$. For any $1\leq i\leq
N$, since
$\langle\nu(\epsilon_i),\nu(\epsilon_i)\rangle=1$, there exist $\iota_i=\pm 1$ and a permutation
$\sigma\in S_N$ such that
$\nu(\epsilon_i)=\iota_i\epsilon_{\sigma(i)}$.
We introduce the following notation for later used
\begin{align} \label{eq:n2}
i_0:=i \text{ and }i_r:=\left(\prod_{p=0}^{r-1}\iota_{\sigma^p(i)}\right) \sigma^r(i)\ 1\leq i\leq N,\ 1\le r\le m-1.
\end{align}

We say an automorphism $\theta$ of $(\mathcal G(Q),\tau)$ is compatible with the isometry $\nu$ if
there exists a function $\eta: \mathbb Z_m\times Q\rightarrow \mathbb C^*$ such that
\begin{align}\label{eq:c2}
\theta^r(e_{\rho_ii,\rho_jj})=\eta(r,\rho_i\epsilon_i-\rho_j\epsilon_j)e_{\rho_ii_r,\rho_jj_r},\ r\in \mathbb Z_m, (\rho_ii,\rho_jj)\in \mathcal{J}.
\end{align}
Given an automorphism $\theta$ of $(\mathcal G(Q),\tau)$ which
satisfies $\theta^{m}=$Id, preserves the bilinear form $\la
,\ra_\mathcal G$, and is compatible with the isometry $\nu$. Then we
have a twisted $\Gamma$-Lie algebra $\widehat{\mathcal
G(Q)}_\tau(\theta,m,\Gamma)$, which is also denoted by $\widehat{\mathcal
G}(Q,\nu,m,\Gamma)$.

Recall the elements $\widetilde{e_{\rho_ii,\rho_jj}}(c,n),
(\rho_ii,\rho_jj)\in \mathcal{J}, n\in \mathbb{Z},c\in \Gamma$ defined in (\ref{eq:e1}), which together with
$\mathbf c$ span the Lie algebra
$\widehat{\mathcal{G}}(Q,\nu,m,\Gamma)$.
For any $(\rho_ii,\rho_jj)\in \mathcal J, c\in \Gamma$ and $n\in \mathbb Z$, we set
$
G_{\rho_ii,\rho_jj}(c,z)=\sum_{n\in \mathbb Z}\widetilde{e_{\rho_ii,\rho_jj}}(c,n)z^{-n}.
$
Then we have the following proposition whose verification
is straightforward.
\begin{prop} \label{prop:cr} Let $(\rho_ii,\rho_jj),(\rho_kk,\rho_ll)\in \mathcal{J},r\in \mathbb Z_m$ and
$c_1,c_2\in \Gamma$, then
\begin{align*}
&G_{\rho_ii,\rho_jj}(c_1,z)=(-1)^{\delta_{ij}}G_{-\rho_jj,-\rho_ii}(c_1^{-1},c_1z),\\
&G_{\rho_ii,\rho_jj}(c_1,\omega^{-r}z)=\eta(r,\rho_i\epsilon_i-\rho_j\epsilon_j)G_{\rho_ii_r,\rho_jj_r}(c_1,z).
\end{align*}
and
\begin{equation}\label{eq:cr}\begin{split}
&\qquad [G_{\rho_ii,\rho_jj}(c_1,z_1),G_{\rho_kk,\rho_ll}(c_2,z_2)]\\
=&\quad m^{-1}\sum_{r\in
\mathbb{Z}_m}\delta_{\rho_ii,-\rho_kk_r}(-1)^{\delta_{ij}}
\xi_r(\underline{\alpha},\underline{\beta})G_{-\rho_jj,\rho_ll_r}(c_1^{-1}c_2,c_1z_1)
\delta(\omega^rz_2/z_1)\\
&+m^{-1}\sum_{r\in \mathbb{Z}_m}\delta_{\rho_jj,\rho_kk_r}
\xi_r(\underline{\alpha},\underline{\beta})G_{\rho_ii,\rho_ll_r}(c_1c_2,z_1)
\delta(\omega^rz_2/c_1z_1)\\
&+m^{-1}\sum_{r\in
\mathbb{Z}_m}\delta_{\rho_ii,\rho_ll_r}(-1)^{\delta_{ij}+\delta_{kl}}
\xi_r(\underline{\alpha},\underline{\beta})G_{-\rho_jj,-\rho_kk_r}(c_1^{-1}c_2^{-1},c_1z_1)
\delta(\omega^rc_2z_2/z_1)\\
&+m^{-1}\sum_{r\in
\mathbb{Z}_m}\delta_{-\rho_jj,\rho_ll_r}(-1)^{\delta_{kl}}
\xi_r(\underline{\alpha},\underline{\beta})G_{\rho_ii,-\rho_kk_r}(c_1c_2^{-1},z_1)
\delta(\omega^rc_2z_2/c_1z_1)\\
&+m^{-2}\sum_{r\in \mathbb{Z}_m}\delta_{c_1c_2,1}
\delta_{\rho_jj,\rho_kk_r}\delta_{\rho_ii,\rho_ll_r}
\xi_r(\underline{\alpha},\underline{\beta})(D\delta)(\omega^rz_2/c_1z_1)\\
&+m^{-2}\sum_{r\in
\mathbb{Z}_m}\delta_{c_1,c_2}(-1)^{\delta_{ij}}\delta_{\rho_ii,-\rho_kk_r}\delta_{\rho_jj,-\rho_ll_r}
\xi_r(\underline{\alpha},\underline{\beta})(D\delta)(\omega^rz_2/z_1),
\end{split}\end{equation}
where $\underline{\alpha}:=\rho_i\epsilon_i-\rho_j\epsilon_j,\
\underline{\beta}:=\rho_k\epsilon_k -\rho_l\epsilon_l$,
$\xi_r(\underline{\alpha},\underline{\beta})=\varepsilon(\underline{\alpha},\nu^r\underline{\beta})\eta(r,\underline{\beta}),$
and $\delta(z)=\sum_{n\in \mathbb Z}z^n,\ (D\delta)(z)=\sum_{n\in
\mathbb Z} nz^n.$ \qed
\end{prop}

\section{Generalized Fock space and twisted $\Gamma$-vertex operators}\label{sec 2}

In this section we define the generalized Fock space and then
construct a family of twisted $\Gamma$-vertex operators. We also
present some properties of these operators. Recall that $P=\mathbb
Z\epsilon_1\oplus \cdots \oplus \mathbb Z\epsilon_N$ is a lattice of
rank $N$, and $(Q,\nu,m,\Gamma)$ is a quadruple satisfying the
assumptions (A1)-(A5).

First we introduce a Heisenberg algebra
$\mathcal H(\nu)$ associated to the pair $(P,\nu)$. Let
$\mathcal{H}=P\otimes_{\mathbb{Z}}\mathbb{C}$. We extend the
isometry $\nu$ of $P$ to a linear automorphism of $\mathcal{H}$, and
extend the bilinear form $\langle , \rangle$ on $P$ to a $\mathbb{C}$-bilinear
non-degenerate symmetric form  on $\mathcal H$. For any $n\in \mathbb Z$ and
$h\in \mathcal H$, we define
$h_{(n)}=\sum_{p\in \mathbb Z_m}\omega^{-np}\nu^p(h)\ \text{and}\
\mathcal H_{(n)}=\{h_{(n)}|h\in \mathcal H\}.$
Define a Heisenberg algebra
$$\mathcal H(\nu)=\mathrm{span}_\mathbb C\{x(n),1| n\in \mathbb Z\setminus \{0\}, x\in \mathcal H_{(n)}\}$$
with Lie bracket
\begin{align*}
[x(n),y(r)]=m^{-1}\la x,y \ra n\delta_{n+r,0},\ n,r\in \mathbb{Z}\setminus \{0\}, x\in \mathcal{H}_{(n)},y\in \mathcal{H}_{(r)}.
 \end{align*}
Let $S:=S(\mathcal H(\nu)^-)$ be the symmetric algebra over the
commutative
 subalgebra  $\mathcal{H}(\nu)^-$ of $\mathcal H(\nu)$ spanned by the elements $x(n)$ for $x\in \mathcal H_{(n)}$ and $n<0$.
 It is well-known that the Heisenberg algebra $\mathcal H(\nu)$ has a canonical irreducible representation on
 the symmetric algebra $S$.
 Let $z,z_1,z_2$ be formal variables and $\alpha \in \mathcal{H}$. We set
 \begin{align*}
  E^{\pm}(\alpha,z)=\exp(-\sum_{\pm n\in \mathbb Z_+}m\frac{\alpha_{(n)}(n)}{n}z^{-n})\in \mbox{End}(S)[[z^{\mp 1}]].
 \end{align*}

 Recall that $\omega$ is a
primitive $m$-th root of unity. Let $\omega_0$ be a primitive
$m_0$-th root of unity such that $\omega_0^{\frac{m_0}{m}}=\omega$,
where $m_0=m$ if $m$ is even, and $m_0=2m$ if $m$ is odd. Let $\la
\omega_0 \ra$ be the cyclic subgroup of $\mathbb C^*$ generated by
$\omega_0$. Then we have $-1,\omega\in \la \omega_0 \ra$. Following
\cite{L}, we define a function $C: Q\times Q\rightarrow
\la\omega_0\ra$ by
\begin{align*} C(\alpha,\beta)=\prod_{p\in \mathbb{Z}_m}(-\omega^{-p})^{\la\alpha,\nu^p\beta\ra},
                     \ \alpha,\beta\in Q.
\end{align*}
Let $\varepsilon_C:Q \times Q\rightarrow \la\omega_0\ra$ be a
normalized $2$-cocycle associated with the function $C$. That is,
$\varepsilon_C$ satisfies the following conditions
   \begin{equation}\label{eq:2.1}\begin{split}
   &\varepsilon_C(\alpha,\beta)\varepsilon_C(\alpha+\beta,\gamma)=
   \varepsilon_C(\beta,\gamma)\varepsilon_C(\alpha,\beta+\gamma),\\
   &\varepsilon_C(0,0)=1,\ \varepsilon_C(\alpha,\beta)/\varepsilon_C(\beta,\alpha)=C(\alpha,\beta),
  \end{split}\end{equation}
 for $\alpha,\beta,\gamma\in Q$.
Now we define a twisted group algebra $\mathbb
C[Q,\varepsilon_C]=\oplus_{\alpha\in Q}\mathbb C e_\alpha$
 with multiplication
$e_\alpha e_\beta=\varepsilon_C(\alpha,\beta) e_{\alpha+\beta},\ \alpha,\beta\in Q.$
Such a twisted group algebra can also be obtained in the following
way. Let $\widehat{Q}$ be the unique (up to equivalence) central
extension of $Q$ by the cyclic group   $\la\kappa_0\ra$ of order
$m_0$ such that
$$aba^{-1}b^{-1}=\prod_{p\in \mathbb{Z}_m}(\kappa_0^{\frac{m_0}{2}
-\frac{m_0p}{m}})^{\la\bar{a},\nu^p\bar{b}\ra},\ a,b\in \widehat{Q},$$
where $\ \bar \ : \widehat{Q}\rightarrow Q$ is the natural
homomorphism. It is known that, by choosing a section of
$\widehat{Q}$, the algebra $\mathbb C[Q,\varepsilon_C]$ is
isomorphic to the quotient algebra $\mathbb C\{Q\}:=\mathbb
C[\widehat{Q}]/(\kappa_0-\omega_0)\mathbb C[\widehat{Q}]$ ($\mathbb
C[\ \cdot \ ]$ denotes the group algebra).

Now we are in a position to give the definition of the generalized
Fock space. Let $\hat{\nu}$ be an automorphism of  $\mathbb
C[Q,\varepsilon_C]$ such that
\begin{align}\label{eq:2..2}
\hat{\nu}(e_\alpha)\in \mathbb C e_{\nu(\alpha)}\ \text{and}\
\hat{\nu}^m=\mathrm{Id}.
\end{align}
Let $T$ be a $\mathbb C[Q,\varepsilon_C]$-module. Following \cite{DL}, we assume that $\mathcal H_{(0)}$ acts on $T$ in such a
way that
$$T=\oplus_{\alpha\in Q} T_{\alpha_{(0)}},\quad
\text{where}\quad T_{\alpha_{(0)}}=\{t\in T|h.t=\la h,\alpha_{(0)}\ra,\ h\in
\mathcal H_{(0)}\},$$ and  that the actions of $\mathbb
C[Q,\varepsilon_C]$ and $\mathcal H_{(0)}$ on $T$ are compatible in
the sense that
\begin{align}\label{eq:2.2}
e_\alpha.b\in T_{(\alpha+\beta)_{(0)}},\ e_\alpha^{-1}
\hat{\nu}(e_\alpha).b=\omega^{-\la\sum \nu^p\alpha,\beta+\frac
\alpha 2\ra}b,
\end{align}
for $\alpha,\beta\in Q$ and $b\in T_{\beta_{(0)}}$.
 The following
remark shows the existence of the automorphism $\hat{\nu}$ and
module $T$.
\begin{rem}  It was proved in  Section 5 of \cite{L} that $\nu$ can be lifted to be an automorphism $\hat{\nu}$ of
$\widehat{Q}$ such that
$$\hat{\nu}(\kappa_0)=\kappa_0,\ \hat{\nu}^m=1\ \text{and}\ (\hat{\nu}(a))^-=\nu(\bar{a}),\ a\in \widehat{Q},$$
which implies that $\hat{\nu}$ induces an automorphism, still called $\hat{\nu}$, of the
quotient algebra $\mathbb C\{Q\}$. Note that $\mathbb C\{Q\} \cong \mathbb
C[Q,\varepsilon_C]$, thus $\hat{\nu}$ is an automorphism of
$\mathbb C[Q,\varepsilon_C]$ satisfying (\ref{eq:2..2}). Next we
show the existence of the module $T$. Set $\mathcal N=\{\alpha\in
Q|\la\alpha,\beta_{(0)}\ra=0, \forall \beta\in Q\}$. Let
$\widehat{\mathcal N}\subset \widehat{Q}$ be the pull back of
$\mathcal N$ in $Q$. In \cite{L}, a certain class of induced
$\widehat{Q}$-module $T=\mathbb C[\widehat{Q}]\otimes_{\mathbb
C[\widehat{N}]}T'$ ($T'$ is an $\widehat{\mathcal N}$ module), on
which $\kappa_0$ acts by $\omega_0$ and
 $a^{-1}\hat{\nu}(a)$ acts by $\omega^{-\la \sum \nu^p \bar{a},\bar{a}\ra/2}$ for $a\in \widehat{Q}$, was constructed.
 We define a $\mathcal H_{(0)}$-action on $T$ by $h.(b\otimes t)=\la h,\bar{b}\ra b\otimes t$
 for $h\in \mathcal H_{(0)},
b\in \widehat{Q}$ and $t\in T'$. One easily checks
$\la \mathcal H_{(0)}, \alpha\ra =0, \alpha\in \mathcal N$, which implies  the
action is well-defined. For $\alpha\in Q$, set
$T_{\alpha_{(0)}}=\text{span}_\mathbb C\{\alpha_{(0)}\otimes
t'|t'\in T'\}$. Then we have $h.t=\la h,\alpha_{(0)}\ra t$ for $h\in
\mathcal H_{(0)},t\in T_{\alpha_{(0)}}$. Finally, it is easy to
check that the action of $\mathcal H_{(0)}$ is compatible with the
natural action of $\mathbb C[Q,\varepsilon_C]$ on $T$. \qed
\end{rem}

Now we define the generalized Fock space $V_T=T\otimes S$ to be the
tensor product of the $\mathbb C[Q,\varepsilon_C]$-module
 $T$ and the $\mathcal H(\nu)$-module $S$.
For $\alpha\in \mathcal H_{(0)}, \beta\in Q$ and $h\in \mathcal
H(\nu)$, we also define certain operators acting on $V_T$ as follows
 \begin{align*}
 \alpha(0).(t\otimes s)=\la\alpha,\gamma\ra t\otimes s,\
e_\beta.(t\otimes s)=(e_\beta.t)\otimes s,\
 h.(t\otimes s)= t\otimes (h.s),
 \end{align*}
  where $t\in T_{\gamma_{(0)}},\gamma\in Q$ and $s\in S$.
  In case $\alpha\in \mathcal H_{(0)}$ such that $\la \alpha,Q\ra \in \mathbb Z$, for  a formal variable $z$ and a nonzero complex number $c$, we define operators
  \begin{align*}
  z^\alpha.(t\otimes s)=z^{\la \alpha,\gamma\ra}t\otimes s,\
  c^\alpha.(t\otimes s)=c^{\la \alpha,\gamma\ra}t\otimes s.
  \end{align*}

For $\alpha\in \mathcal H$, we set
$\alpha(z)=\sum_{n\in \mathbb Z}\alpha_{(n)}(n)z^{-n}.$
The following facts about the operators defined above can be easily
verified.

\begin{lem}\label{lem:2.2}
  For $\alpha,\beta\in Q, \gamma, \gamma_1,\gamma_2\in \mathcal H,c\in \Gamma$ and
  $r\in \mathbb{Z}_m$, we have
  \begin{align*}
  &\hat{\nu}e_\alpha=e_\alpha\omega^{-\sum \nu^p\alpha-\la\alpha,
  \sum\nu^p\alpha\ra/2},\ \nu^r\gamma(z)=\gamma(\omega^{-r}z),\ E^\pm(\nu^r\gamma,z)=E^\pm(\gamma,\omega^{-r}z),\\
  &z^{\sum\nu^p\gamma} e_\beta=e_\beta z^{\sum \nu^p\gamma+\sum\la\nu^p\gamma,
  \beta\ra},\ c^{\sum\nu^p\gamma} e_\beta=e_\beta c^{\sum\nu^p\gamma+\sum\la\nu^p\gamma,\beta\ra},\\
  &e_\alpha e_\beta=C(\alpha,\beta)e_\beta
  e_\alpha,\ [\gamma_{(0)}(0),e_\beta]=m^{-1}\la\sum \nu^p\gamma,\beta\ra e_\beta,\\
  &[\gamma_1(z_1),E^{\pm}(\gamma_2,z_2)]
  =\sum_{p\in \mathbb{Z}_m}\frac{\la\gamma_1,\nu^p\gamma_2\ra}{m}
  \big(\sum_{\mp n>0}(\omega^pz_2/z_1)^nE^{\pm}(\gamma_2,z_2)\big),\\
  &E^+(\gamma_1,z_1)E^-(\gamma_2,z_2)=E^-(\gamma_2,z_2)E^+(\gamma_1,z_1)
  \prod_{p\in \mathbb{Z}_m}(1-\omega^pz_2/z_1)^{\la\gamma_1,\nu^p\gamma_2\ra}.\qed
  \end{align*}
\end{lem}

Before giving the twisted $\Gamma$-vertex operators, we define two
constants:
  \begin{equation*}\begin{split}
  \zeta(\alpha)&=\begin{cases}\zeta'(\alpha)2^{\la\alpha,\nu^{m/2}\alpha\ra/2},\ &\text{if}\ m\in 2\mathbb{Z},\\
  \zeta'(\alpha),\ &\text{if}\ m\in 2\mathbb{Z}+1, \end{cases}\\
  \kappa(\rho_ii,\rho_jj,c)&=\prod_{0\leq p<m}(1-c\omega^{p})^{-\la\rho_i\epsilon_i,\rho_j\nu^p\epsilon_j\ra}
  \prod_{0<p<m}(1-\omega^p)^{\la\rho_i\epsilon_i,\rho_j\nu^p\epsilon_j\ra},
   \end{split}\end{equation*}
  where $\alpha\in Q, \zeta'(\alpha)=\prod_{0<p<m/2}(1-\omega^{p})^{\la\alpha,\nu^p\alpha\ra},$
and $(\rho_ii,\rho_jj)\in \mathcal J$(see (\ref{eq:1.1})) $,c\in
\Gamma$ such that either $\rho_ii\ne \rho_jj$ or $\rho_ii=\rho_jj, c\ne 1$.

Now we define the twisted $\Gamma$-vertex operators
$Y_{\rho_ii,\rho_jj}(c,z)$ on $V_T$ by
\begin{equation*}\begin{aligned}
Y_{\rho_ii,\rho_jj}(c,z)=
\begin{cases}
\rho_i\epsilon_i(z),& \text{if }\rho_ii=\rho_jj, c=1,\\
m^{-1}\zeta(\rho_i\epsilon_i-\rho_j\epsilon_j)
\kappa(\rho_ii,\rho_jj,c)X_{\rho_ii,\rho_jj}(c,z),\
&\text{otherwise,}
\end{cases}\end{aligned}\end{equation*}
where $(\rho_ii,\rho_jj)\in \mathcal{J}$, $c\in \Gamma$, and
$X_{\rho_ii,\rho_jj}(c,z)$ is defined as follows
\begin{align*}
 X_{\rho_ii,\rho_jj}(c,z)=&e_{\rho_i\epsilon_i-\rho_j\epsilon_j}
E^-(\rho_i\epsilon_i,z)E^-(-\rho_j\epsilon_j,cz)E^+(\rho_i\epsilon_i,z)E^+(-\rho_j\epsilon_j,cz)\\
&\cdot z^{\sum \nu^p(\rho_i\epsilon_i-\rho_j\epsilon_j)+\sum
\la\rho_i\epsilon_i-\rho_j\epsilon_j,\nu^p(\rho_i\epsilon_i-\rho_j\epsilon_j)\ra/2}
c^{-\sum \nu^p\rho_j\epsilon_j+\sum
\la\epsilon_j,\nu^p\epsilon_j\ra/2}.
\end{align*}


 In what follows we describe the relations among the
twisted $\Gamma$-vertex operators $Y_{\rho_ii,\rho_jj}(c,z)$
defined above.
Let $\eta(r,\alpha)$ be the complex numbers
 determined by
\begin{align}\label{eq:c3}
\hat{\nu}^r(e_\alpha)=\eta(r,\alpha)e_{\nu^r(\alpha)},\ r\in
\mathbb Z_m, \alpha\in Q.
\end{align}
Recall the notation $i_r$ introduced in \eqref{eq:n2}. Then by definition we have $
\zeta(\nu^r(\alpha))=\zeta(\alpha)$ and $
\kappa(\rho_ii_r,\rho_jj_r,c)=\kappa(\rho_ii,\rho_jj,c)$.
It follows from this and the first three identities in Lemma \ref{lem:2.2} that
\begin{prop}\label{prop:2.3} For $(\rho_ii,\rho_jj)\in \mathcal{J}, c\in \Gamma$
and $r\in \mathbb{Z}_m$, we have
\begin{align}
&Y_{\rho_ii,\rho_jj}(c,z)=(-1)^{\delta_{ij}}Y_{-\rho_jj,-\rho_ii}(c^{-1},cz) \label{eq:R1},\\
&Y_{\rho_ii,\rho_jj}(c,\omega^{-r}z)=\eta(r,\rho_i\epsilon_i-\rho_j\epsilon_j)Y_{\rho_ii_r,\rho_jj_r}(c,z).\label{eq:R2}\qed
\end{align}
\end{prop}

\section{An identity}\label{sec 3}
In this section we prove an identity (see \eqref{eq:3.15}) which will be used in the next section to calculate the commutator relations for the twisted $\Gamma$-vertex operators. As we have mentioned that this identity is a nontrivial generalization of the combinatorial identity given by Lepowsky in [L]. Throughout this section, the notations $s$ and $t_i, i=1,2,\cdots$ denote some  distinct nonzero complex numbers.

First we recall some well-known identities:

\begin{align}
 \sum_{i=1}^n\left(\prod_{j\not=i}^n\frac
 {t_i}{t_i-t_j}\right)\frac {s} {s-t_i}=\prod^n_{i=1}\frac {s} {s-t_i},\ \sum_{i=1}^n\left(\prod_{j\not=i}^n\frac
 {t_j}{t_i-t_j}\right)\frac {t_i} {s-t_i}=
 \prod^n_{i=1}\frac {t_i}{s-t_i},\label{eq:3.1}
 \end{align}
\begin{align}
 \sum_{i=1}^n&\left(\prod_{j\not=i}^n \frac
 {t_i}{t_i-t_j}\right)\left(\frac {s} {s-t_i}\right)^2
 =\left(1+\sum_{j=1}^n\frac{t_j}{s-t_j}\right)\prod^n_{i=1}\frac {s} {s-t_i}.\label{eq:3.3}\\
  \sum_{i=1}^n&\left(\prod_{j\not=i}^n \frac{t_j}{t_i-t_j}\right)\left(\frac {t_i} {s-t_i}\right)^2=
\left(\sum_{j=1}^n\frac{s}{s-t_j}-1\right) \prod^n_{i=1}\frac {t_i}{s-t_i}.\label{eq:3.4}
\end{align}

Now we use these identities to prove the following result:

\begin{lem}\label{lem:3.2} Let $a_i=1$ or $2$ for $i\in I:=\{1,\cdots, n\}$.
 Set $I(t)=\{i\in I|a_i=t\}$ for $t=1$ or $2$. Then we have
 \begin{equation}\begin{split}\label{eq:3.5}
 &\prod_{i=1}^n\left(\frac{s}{s-t_i}\right)^{a_i}
 =\sum_{i\in I(1)}\left(\prod_{j\in I,j\neq i}\left(\frac{t_i}{t_i-t_j}\right)^{a_j}\right)\frac{s}{s-t_i}\\
 &+\sum_{i\in I(2)}\left(\prod_{j\in I,j\ne i}\left(\frac{t_i}{t_i-t_j}\right)^{a_j}\right)
 \left[\frac{st_i}{(s-t_i)^2}+
 \left(1+\sum_{j\in I,j\ne i} \frac{a_jt_j}{t_j-t_i}\right)\frac{s}{s-t_i}\right],\end{split}\end{equation}
\begin{equation}\begin{split} \label{eq:3.6} &\prod_{i=1}^n\left(\frac{t_i}{s-t_i}\right)^{a_i}
 =\sum_{i\in I(1)}\left(\prod_{j\in I,j\neq i}\left(\frac{t_j}{t_i-t_j}\right)^{a_j}\right)\frac{t_i}{s-t_i}\\
 &+\sum_{i\in I(2)}\left(\prod_{j\in I,j\ne i}\left(\frac{t_j}{t_i-t_j}\right)^{a_j}\right)
 \left[\frac{st_i}{(s-t_i)^2}+
 \left(\sum_{j\in I,j\ne i} \frac{a_jt_i}{t_j-t_i}-1\right)\frac{t_i}{s-t_i}\right].
 \end{split}\end{equation}
 In particular, let $z$ be a formal variable, one has 
\begin{equation}\begin{split}\label{eq:3.7}
 &\prod_{i=1}^n\left(\frac{1}{1-t_iz}\right)^{a_i}
 =\sum_{i\in I(1)}\left(\prod_{j\in J,j\neq i}\left(\frac{t_i}{t_i-t_j}\right)^{a_j}\right)\frac{1}{1-t_iz}\\
 &+\sum_{i\in I(2)}\left(\prod_{j\in J,j\ne i}\left(\frac{t_i}{t_i-t_j}\right)^{a_j}\right)
 \left[\frac{zt_i}{(1-t_iz)^2}+
 \left(1+\sum_{j\in J,j\ne i} \frac{a_jt_j}{t_j-t_i}\right)\frac{1}{1-t_iz}\right],
 \end{split}\end{equation}
 \begin{equation}\begin{split}\label{eq:3.8}
 &\prod_{i=1}^n\left(\frac{t_iz}{1-t_iz}\right)^{a_i}
 =\sum_{i\in I(1)}\left(\prod_{j\in I,j\neq i}\left(\frac{t_j}{t_i-t_j}\right)^{a_j}\right)\frac{t_iz}{1-t_iz}\\
 &+\sum_{i\in I(2)}\left(\prod_{j\in I,j\ne i}\left(\frac{t_j}{t_i-t_j}\right)^{a_j}\right)
 \left[\frac{t_iz}{(1-t_iz)^2}+
 \left(\sum_{j\in I,j\ne i} \frac{a_jt_i}{t_j-t_i}-1\right)\frac{t_iz}{1-t_iz}\right].
 \end{split}\end{equation}
 \end{lem}
\begin{proof}  For the identity \eqref{eq:3.5}, by applying the first identity in (\ref{eq:3.1}), we have
 \begin{align*}
 \prod_{i=1}^n\left(\frac{s}{s-t_i}\right)^{a_i}
 =\left(\sum_{i\in I}D_i\frac{s}{s-t_i}\right)\left(\sum_{k\in I(2)}E_k\frac{s}{s-t_k}\right)
 =G_1+G_2+G_3+G_4,
 \end{align*}
 where $D_i=\prod_{j\in I,j\ne i}\frac{t_i}{t_i-t_j}$,
 $E_k=\prod_{j\in I(2),j\ne k}\frac{t_k}{t_k-t_j},$ and
 \begin{align*}
 G_1&=\sum_{i\in I(1),k\in I(2)}D_iE_k \frac{t_i}{t_i-t_k}\frac{s}{s-t_i},\
 G_2=\sum_{i,k\in I(2),i\ne k}D_iE_k \frac{t_i}{t_i-t_k}\frac{s}{s-t_i},\\
 G_3&=\sum_{i\in I,k\in I(2),i\ne k} D_iE_k \frac{t_k}{t_k-t_i}\frac{s}{s-t_k},\
 G_4=\sum_{i\in I(2)}D_iE_i\left(\frac{s}{s-t_i}\right)^2.\end{align*}

By using the first identity in (\ref{eq:3.1}) and identity (\ref{eq:3.3}), we get
\begin{align*}G_1=&\sum_{i\in I(1)}D_i\left(\prod_{k\in I(2)}\frac{t_i}{t_k-t_i}\right)
 \frac{s}{s-t_i}
 =\sum_{i\in I(1)}D_iE_i\frac{s}{s-t_i},\\
 G_2=&\sum_{i\in I(2)}D_i\left(
 \sum_{k\in I(2),k\ne i}\left(\prod_{j\in I(2),j\ne i,k}\frac{t_k}{t_k-t_j}\right)\frac{t_k}{t_k-t_i}
 \frac{t_i}{t_i-t_k}\right)\frac{s}{s-t_i}\\
 =&\sum_{i\in I(2)}D_i\left[
 \sum_{k\in I(2),k\ne i}\left(\prod_{j\in I(2),j\ne i,k}\frac{t_k}{t_k-t_j}\right)\left(
 \frac{t_i}{t_i-t_k}-\left(\frac{t_i}{t_i-t_k}\right)^2\right)\right]\frac{s}{s-t_i}\\
 =&\sum_{i\in I(2)}D_iE_i\left(\sum_{k\in I(2),k\ne i}
 \frac{t_k}{t_k-t_i}\right)\frac{s}{s-t_i},\\
 G_3=&\sum_{i\in I(2)}E_i\left(\sum_{k\in I,k\ne i}
\prod_{j\in I,j\ne k}\left(\frac{t_k}{t_k-t_j}\right)\left(
\frac{A_k}{t_k-t_i}\frac{t_i}{t_i-t_k}\right)\right)\frac{s}{s-t_i}\\
 =&\sum_{i\in I(2)}E_i\left[\sum_{k\in I,k\ne i} \left(\prod_{j\in
I,j\ne k}\frac{t_k}{t_k-t_j}\right)\left(
 \frac{t_i}{t_i-t_k}-\left(\frac{t_i}{t_i-t_k}\right)^2
\right)\right]\frac{s}{s-t_i}\\
=&\sum_{i\in I(2)}D_iE_i\left(\sum_{j\in I,j\ne
i}\frac{t_j}{t_j-t_i}\right)\frac{s}{s-t_i}.
\end{align*}

It is clear that the expression of $G_1+G_2+G_3+G_4$ coincides with the right hand-side of \eqref{eq:3.5},
 as desired. The identity \eqref{eq:3.6} can be proved similarly by using the second identity in (\ref{eq:3.1}) and identity (\ref{eq:3.4}), which is omitted.
\end{proof}

According to the formal power identities \eqref{eq:3.7} and \eqref{eq:3.8}, we have

\begin{lem}\label{lem:3.4} Let $I,a_i,i\in I, I(t)$ be the same as in Lemma \ref{lem:3.2}.
Then
\begin{align*}
&\prod_{i=1}^n\left(\frac{1}{1-t_iz}\right)^{a_i}-\prod_{i=1}^n
\left(\frac{-t_i^{-1}z^{-1}}{1-t_i^{-1}z^{-1}}\right)^{a_i}=
\sum_{i\in I(1)}\prod_{j\in I,j\ne i}\left(\frac{t_i}{t_i-t_j}\right)^{a_j}\delta(t_iz)\\&
+\sum_{i\in I(2)}\prod_{j\in I,j\ne
i}\left(\frac{t_i}{t_i-t_j}\right)^{a_j}
\left[(D\delta)(t_iz)+\left(1+\sum_{j\in I,j\ne
i}\frac{a_jt_j}{t_j-t_i}\right)\delta(t_iz)\right].
\end{align*}
\end{lem}
\begin{proof}
From the formal power identities  \eqref{eq:3.7} and \eqref{eq:3.8} , we know
 \begin{align*}
 &\prod_{i=1}^n\left(\frac{1}{1-t_iz}\right)^{a_i}-\prod_{i=1}^n
\left(\frac{-t_i^{-1}z^{-1}}{1-t_i^{-1}z^{-1}}\right)^{a_i}\\
=&\sum_{i\in I(1)}\prod_{j\in I,j\ne
i}\left(\frac{t_i}{t_i-t_j}\right)^{a_j}
\left(\frac{1}{1-t_iz}+\frac{t_i^{-1}z^{-1}}{1-t_i^{-1}z^{-1}}\right)\\
& +\sum_{i\in I(2)}\prod_{j\in I,j\ne
i}\left(\frac{t_i}{t_i-t_j}\right)^{a_j}
\left(\frac{t_iz}{(1-t_iz)^2}-\frac{t_i^{-1}z^{-1}}{(1-t_i^{-1}z^{-1})^2}\right)\\
&+\sum_{i\in I(2)}\prod_{j\in I,j\ne i}\left(\frac{t_i}{t_i-t_j}\right)^{a_j}
\left(\sum_{j\in I,j\ne i}\frac{a_jt_j}{t_j-t_i}+1\right)
\left(\frac{1}{1-t_iz}+\frac{t_i^{-1}z^{-1}}{1-t_i^{-1}z^{-1}}\right)\\
 =&\sum_{i\in I(1)}\prod_{j\in I,j\ne
i}\left(\frac{t_i}{t_i-t_j}\right)^{a_j} \delta(t_iz)\\
&+\sum_{i\in I(2)}\prod_{j\in I,j\ne i}\left(\frac{t_i}{t_i-t_j}\right)^{a_j}
\left[(D\delta)(t_iz)+\left(\sum_{j\in I,j\ne
i}\frac{a_jt_j}{t_j-t_i} +1\right)\delta(t_iz)\right]
\end{align*}
as desired. \end{proof}

Let $f(z)\in \mathbb{C}[z,z^{-1}]$ and $a\in \mathbb{C}^*$. The
following formal power series identities (cf. \cite{FLM}) are well-known.
\begin{align} f(z)\delta(az)&=f(a^{-1})\delta(az),\label{eq:3.13}\\
f(z)(D\delta)(az)&=f(a^{-1})(D\delta)(az)-(D_zf)(a^{-1})\delta(az),\label{eq:3.14}
\end{align}
where $D_z=z\frac{d}{dz}$. Now we state and prove our main result in this section.
\begin{prop}\label{prop:3.6} Let $t_i,i\in I$ be distinct nonzero complex numbers
and $I=\{1,\cdots,n\}$. Let $a_i\in \mathbb{Z}$ and $a_i\geq -2,$
 $i\in I$. For $t\in \mathbb{Z}$ and $t\geq -2$, set $I(t)=
\{i\in I|a_i=t\}$. Then we have
\begin{equation}\label{eq:3.15}\begin{split}
&\prod_{i\in I}(1-t_iz)^{a_i}-\prod_{i\in I}(1-t_i^{-1}z^{-1})^{a_i}(-t_iz)^{a_i}
=\sum_{i\in I(-1)}\prod_{j\in I,j\ne
i}(1-t_jt_i^{-1})^{a_j}\delta(t_iz)\\&+
\sum_{i\in I(-2)}\prod_{j\in I,j\ne i}(1-t_jt_i^{-1})^{a_j} [(D\delta)(t_iz)+(1+\sum_{j\in I,j\ne
i}a_j(t_it_j^{-1}-1)^{-1})\delta(t_iz)].
\end{split}
\end{equation}
\end{prop}

\begin{proof} Note that $(1-t_iz)^{a_i}=(1-t_i^{-1}z^{-1})^{a_i}(-t_iz)^{a_i}$ if
$a_i\ge 0$. Set $I_+=\{i\in I| a_i\ge 0\}$, and $I_-=I(-1)\cup
I(-2)$.
 From Lemma \ref{lem:3.4} and (\ref{eq:3.13}), (\ref{eq:3.14}), one has that
\begin{align*}
&\prod_{i\in I}(1-t_iz)^{a_i}-\prod_{i\in I}(1-t_i^{-1}z^{-1})^{a_i}(-t_iz)^{a_i}\\
=&(\prod_{k\in I_+}(1-t_kz)^{a_k})\bigg\{\sum_{i\in
I(-1)}\prod_{j\in I_-,j\ne i}(1-t_jt_i^{-1})^{a_j}
\delta(t_iz)+\\
&\sum_{i\in I(-2)}\prod_{j\in I_-,j\ne
i}(1-t_jt_i^{-1})^{a_j}\big[(D\delta)(t_iz) +(\sum_{j\in I_-,j\ne
i}a_j(t_it_j^{-1}-1)^{-1}+1)\delta(t_iz)\big]\bigg\}
\end{align*}
\begin{align*}
=&\sum_{i\in I(-1)}\prod_{j\in I,j\ne
i}(1-t_jt_i^{-1})^{a_j}\delta(t_iz)
+\sum_{i\in I(-2)}\prod_{j\in I,j\ne i}(1-t_jt_i^{-1})^{a_j}(D\delta)(t_iz)+\\
&\sum_{i\in I(-2)}\prod_{j\in I,j\ne i}(1-t_jt_i^{-1})^{a_j}
(\sum_{j\in I_-,j\ne i}a_j(t_it_j^{-1}-1)^{-1}+1)\delta(t_iz)+\\
&(\sum_{j\in I_+}a_j(1-t_jz)^{-1}(-t_jz))\prod_{k\in
I_+}(1-t_kz)^{a_k}
[\sum_{i\in I(-2)}\prod_{j\in I_-,j\ne i}(1-t_jt_i^{-1})^{a_j}\delta(t_iz)]\\
=&\sum_{i\in I(-1)}\prod_{j\in I,j\ne i}(1-t_jt_i^{-1})^{a_j}\delta(t_iz)+\\
&\sum_{i\in I(-2)}\prod_{j\in I,j\ne i}(1-t_jt_i^{-1})^{a_j}
[(D\delta)(t_iz)+(1+\sum_{j\in I,j\ne
i}a_j(t_it_j^{-1}-1)^{-1})\delta(t_iz)]
\end{align*}
as desired.
\end{proof}

\begin{rem}   If one takes $t_i=\omega^{-i},
a_i=\la \nu^i\alpha,\beta \ra$ for $1\le i\le m, \alpha,\beta\in Q$
in (\ref{eq:3.15}), and notes that
  $$\sum_{p\in \mathbb Z_m, p\ne r}\la \nu^p\alpha,\beta\ra=
\sum_{p\in \mathbb Z_m, p \ne
0}\frac{\la(\nu^p+\nu^{-p})\nu^r\alpha,\beta\ra}{1-\omega^{p}},$$
where $r\in \mathbb Z_m$ satisfies $\la\nu^r\alpha,\beta\ra=-2$,
then the identity (\ref{eq:3.15}) becomes the identity given in
Proposition 4.1 of \cite{L}.
\end{rem}

\section{Vertex operator realizations of twisted $\Gamma$-Lie algebras} \label{sec 4}
 In this section we first state the main result of this paper, and then prove it by calculating the commutator relations for the twisted $\Gamma$-vertex operators
 $Y_{\rho_ii,\rho_jj}(c_1,z_1)$ and $Y_{\rho_kk,\rho_ll}(c_2,z_2)$,
where $(\rho_ii,\rho_jj),(\rho_kk,\rho_ll)\in \mathcal{J}$ and
$c_1,c_2\in \Gamma$.
 For the sake of convenience, we write
\begin{align*}
\underline{\alpha}:=\rho_i\epsilon_i-\rho_j\epsilon_j,\
\underline{\beta}:=\rho_k\epsilon_k -\rho_l\epsilon_l.\end{align*}
\subsection{The main theorem}
For $\alpha,\beta\in Q$, set
\begin{align*}
\varepsilon'(\alpha,\beta)=\prod_{-m/2<p<0} (-\omega^p)^{\la\alpha,\nu^p\beta\ra},\
\varepsilon(\alpha,\beta)=\varepsilon'(\alpha,\beta)\varepsilon_C(\alpha,\beta).
\end{align*}
It follows from (\ref{eq:2.1}) that $\varepsilon$ is a normalized
2-cocycle on $Q$ associated with the function $(-1)^{\la
\alpha,\beta\ra}$ (cf. (\ref{eq:c1})). Then by the lattice construction given in Section 2.4, we
have an involutive associative algebra $(\mathcal G(Q),\tau)$ and a
bilinear form $\la , \ra_\mathcal G$.

Recall the constants $\eta(r,\alpha), r\in \mathbb Z_m,\alpha\in Q$ defined in (\ref{eq:c3}). We define a linear map
$\bar{\nu}$ on $\mathcal G(Q)$ as follows
$$\bar{\nu}(e_{\rho_ii,\rho_jj})=\eta(1,\rho_i\epsilon_i-\rho_j\epsilon_j)e_{\rho_ii_1,\rho_jj_1},\
 (\rho_ii,\rho_jj)\in \mathcal{J}.$$

\begin{lem}
The linear map $\bar{\nu}$ is an  automorphism of the involutive
associative algebra $(\mathcal G(Q),\tau)$, and $\bar{\nu}^{m}=$Id.
Moreover, $\bar{\nu}$ preserves the form $\la ,\ra_\mathcal G$.
\end{lem}

 \begin{proof} Recall that $\hat{\nu}$  is an automorphism  of
$\mathbb{C}[Q,\varepsilon_C]$ (see \eqref{eq:2..2}). Then we have
\begin{align*}
\varepsilon_C(\alpha,\beta)\eta(r,\alpha+\beta)e_{\nu^r(\alpha+\beta)}
=\eta(r,\alpha)\eta(r,\beta)\varepsilon_C(\nu^r(\alpha),\nu^r(\beta))e_{\nu^r(\alpha+\beta)}.
\end{align*}
This, together with the fact
$\varepsilon'(\alpha,\beta)=\varepsilon'(\nu^r\alpha,\nu^r\beta)$,
gives
\begin{align*}
\varepsilon(\alpha,\beta)\eta(r,\alpha+\beta)=\varepsilon(\nu^r(\alpha),\nu^r(\beta))\eta(r,\alpha)\eta(r,\beta),
\ \forall \alpha,\beta \in Q, r\in \mathbb Z_m.\end{align*}
One can conclude from this fact that $\bar{\nu}$ is an automorphism of
$(\mathcal G(Q),\tau)$ and preserves the form $\la ,\ra_\mathcal G$.
\end{proof}
By definition and the previous lemma, we see that the automorphism  $\bar{\nu}$ of
$(\mathcal G(Q),\tau)$ is compatible with the isometry $\nu$ (cf. (\ref{eq:c2})). Therefore, we have
a twisted $\Gamma$-Lie algebra $\widehat{\mathcal G}(Q,\nu,m,\Gamma)$ with Lie bracket given in Proposition \ref{prop:cr}.

Now we state our main theorem of this paper. For $(\rho_ii,\rho_jj)\in \mathcal J$ and $c\in \Gamma$,
let $y_{\rho_ii,\rho_jj}(c,n)$ be the $-n$-component of $Y_{\rho_ii,\rho_jj}(c,z)$. Set $Q'=\{\alpha\in Q:
\la\alpha,\alpha\ra=2\},$ and  $Q''=\{\rho_i\epsilon_i-\rho_j\epsilon_j:1\le i,j\le
N,\rho_i,\rho_j=\pm 1\}\cap Q.$ Then we have
\begin{thm} \label{thm:4.1} The generalized Fock space $V_T$ affords a representation of the twisted $\Gamma$-Lie algebra
$\widehat{\mathcal G}(Q,\nu,m,\Gamma)$ with the actions given by
\begin{align*} \widetilde{e_{\rho_ii,\rho_jj}}(c,n)\mapsto y_{\rho_ii,\rho_jj}(c,n),\ \mathbf c\mapsto 1,
\end{align*}
where $(\rho_ii,\rho_jj)\in \mathcal J, c\in \Gamma$ and $n\in \mathbb Z$. Moreover, if $\Gamma\ne \{1\}$ and $\mathrm{span}_\mathbb
ZQ''=Q$, or if $\Gamma=\{1\}$ and $\mathrm{span}_\mathbb
ZQ'=Q$. Then the $\widehat{\mathcal G}(Q,\nu,m,\Gamma)$-module
$V_T$ is irreducible if and only if the $\mathbb
C[Q,\varepsilon_C]$-module $T$ is irreducible.
\end{thm}

\subsection{Somes lemmas}
To prove Theorem \ref{thm:4.1}, we need the following four lemmas.  The first lemma is well-known (cf. \cite{FLM}).
\begin{lem}\label{lem:4.2}
Let $F(z_1,z_2)$  be a formal power series in $z_1,z_2$ with
coefficients in a vector space such that $\lim_{z_2\rightarrow
cz_1}F(z_1,z_2)$ exists for some $c\in \mathbb C^\ast$. Then

\begin{equation}\label{eq:4.1}
F(z_1,z_2)\delta(z_2/cz_1)=F(z_1,cz_1)\delta(z_2/cz_1),
\end{equation}
\begin{equation}\label{eq:4.2}
F(z_1,z_2)(D\delta)(z_2/cz_1)=F(z_1,cz_1)(D\delta)(z_2/cz_1)+
(D_{z_2}F)(z_1,z_2)\delta(z_2/cz_1).
\end{equation}
\end{lem}

The following two lemmas can be shown by applying  Lemma \ref{lem:2.2} and a straightforward computation.
\begin{lem} \label{lem:4.3.}
$$[X_{\rho_ii,\rho_jj}(c_1,z_1),X_{\rho_kk,\rho_ll}(c_2,z_2)]
=B(z_1,z_2)\cdot C(z_1,z_2),$$
where
\begin{align*}
 &B(z_1,z_2)
 =E^-(\rho_i\epsilon_i,z_1)E^-(-\rho_j\epsilon_j,c_1z_1)
E^-(\rho_k\epsilon_k,z_2)E^-(-\rho_l\epsilon_l,c_2z_2)E^+(\rho_i\epsilon_i,z_1)\\
&  E^+(-\rho_j\epsilon_j,c_1z_1)
E^+(\rho_k\epsilon_k,z_2)E^+(-\rho_l\epsilon_l,c_2z_2) e_{\underline{\alpha}} e_{\underline{\beta}}
z_1^{\la\underline{\beta},\sum\nu^p\underline{\alpha}\ra}
 c_1^{\la\underline{\beta},-\rho_j\sum\nu^p\epsilon_j\ra}
\\& z_1^{\sum\nu^p\underline{\alpha}+\la\underline{\alpha},\sum\nu^p\underline{\alpha}\ra/2}
z_2^{\sum
\nu^p\underline{\beta}+\la\underline{\beta},\sum\nu^p\underline{\beta}\ra/2}
c_1^{-\rho_j\sum\nu^p\epsilon_j+\la\epsilon_j,\sum\nu^p\epsilon_j\ra/2}
c_2^{-\rho_l\sum\nu^p\epsilon_l+\la\epsilon_l,\sum\nu^p\epsilon_l\ra/2},
\end{align*}
and
\begin{align*}
&C(z_1,z_2)
=\prod_{p\in
\mathbb{Z}_m}(1-\omega^pz_2/z_1)^{\rho_i\rho_k\la\epsilon_i,\nu^p\epsilon_k\ra}
(1-\omega^pz_2/c_1z_1)^{-\rho_j\rho_k\la\epsilon_j,\nu^p\epsilon_k\ra}\\
&\cdot
(1-\omega^pc_2z_2/z_1)^{-\rho_i\rho_l\la\epsilon_i,\nu^p\epsilon_l\ra}
(1-\omega^pc_2z_2/c_1z_1)^{\rho_j\rho_l\la\epsilon_j,\nu^p\epsilon_l\ra}\\
&-\prod_{p\in
\mathbb{Z}_m}(1-z_1/\omega^pz_2)^{\rho_i\rho_k\la\epsilon_i,\nu^p\epsilon_k\ra}
(-\omega^pz_2/z_1)^{\rho_i\rho_k\la\epsilon_i,\nu^p\epsilon_k\ra}\\
&\cdot(1-c_1z_1/\omega^pz_2)^{-\rho_j\rho_k\la\epsilon_j,\nu^p\epsilon_k\ra}
(-\omega^pz_2/c_1z_1)^{-\rho_j\rho_k\la\epsilon_j,\nu^p\epsilon_k\ra}\\
&\cdot
(1-z_1/\omega^pc_2z_2)^{-\rho_i\rho_l\la\epsilon_i,\nu^p\epsilon_l\ra}
(-\omega^pc_2z_2/z_1)^{-\rho_i\rho_l\la\epsilon_i,\nu^p\epsilon_l\ra}\\
&\cdot
(1-c_1z_1/\omega^pc_2z_2)^{\rho_j\rho_l\la\epsilon_j,\nu^p\epsilon_l\ra}
(-\omega^pc_2z_2/c_1z_1)^{\rho_j\rho_l\la\epsilon_j,\nu^p\epsilon_l\ra}. \qed
\end{align*}
\end{lem}

\begin{lem}\label{lem:4.3} Let $r\in \mathbb Z_m$,
if $\la\rho_i\epsilon_i,\rho_k\nu^r\epsilon_k\ra=-1$. Then
\begin{equation}\label{eq:4.3}\begin{split}
B(z_1,z_2)|_{z_2=
\omega^{-r}z_1}=c_1^{\la\rho_i\epsilon_i,\rho_j\sum\nu^p\epsilon_j\ra}
\varepsilon_C(\underline{\alpha},\nu^r\underline{\beta})
\eta(r,\underline{\beta})X_{-\rho_jj,\rho_ll_r}(c_1^{-1}c_2,c_1z_1);
\end{split}\end{equation}

If $\la\rho_j\epsilon_j,\rho_k\nu^r\epsilon_k\ra=1$, then
\begin{equation}\label{eq:4.4}
B(z_1,z_2)|_{z_2= \omega^{-r}c_1z_1}=
\varepsilon_C(\underline{\alpha},\nu^r\underline{\beta})
\eta(r,\underline{\beta})X_{\rho_ii,\rho_ll_r}(c_1c_2,z_1);
\end{equation}

If $\la\rho_i\epsilon_i,\rho_l\nu^r\epsilon_l\ra=1$, then
\begin{equation}\label{eq:4.5}\begin{split}
B(z_1,z_2)|_{z_2= \omega^{-r}c_2^{-1}z_1}=&
c_1^{\la\rho_i\epsilon_i,\rho_j\sum\nu^p\epsilon_j\ra}c_2^{\la\rho_k\epsilon_k,\rho_l\sum\nu^p\epsilon_l\ra}
\varepsilon_C(\underline{\alpha},\nu^r\underline{\beta})
\\&\cdot \eta(r,\underline{\beta})X_{-\rho_jj,-\rho_kk_r}(c_1^{-1}c_2^{-1},c_1z_1);
\end{split}\end{equation}

If $\la\rho_j\epsilon_j,\rho_l\nu^r\epsilon_l\ra=-1$, then
\begin{equation}\label{eq:4.6}\begin{split}
B(z_1,z_2)|_{z_2= \omega^{-r}c_1c_2^{-1}z_1}=
c_2^{\la\rho_k\epsilon_k,\rho_l\sum\nu^p\epsilon_l\ra}
\varepsilon_C(\underline{\alpha},\nu^r\underline{\beta})
\eta(r,\underline{\beta})X_{\rho_ii,-\rho_kk_r}(c_1c_2^{-1},z_1).
\end{split}\end{equation}\qed
\end{lem}

Finally, by definition, one immediate has
\begin{lem}\label{lem:4.4} Let $\alpha,\beta\in Q$ and $r\in \mathbb Z_m$, then
\begin{equation}\label{eq:4.7}
\zeta(\alpha)\zeta(\beta) \varepsilon'(\alpha,\nu^r\beta)^{-1}
\zeta(\alpha+\nu^r\beta)^{-1}
=\prod_{0<p<m}(1-\omega^p)^{-\la\alpha,\nu^p(\nu^r\beta)\ra}. \qed
\end{equation}
\end{lem}

\subsection{Proof of Theorem \ref{thm:4.1}} To prove the first part of the theorem, we only need to
prove that the commutator relation for the twisted $\Gamma$-vertex
operators $Y_{\rho_ii,\rho_jj}(c_1,z_1)$ with
$Y_{\rho_kk,\rho_ll}(c_2,z_2)$ is the same as (\ref{eq:cr}) under
the correspondence $Y_{\rho_ii,\rho_jj}(c,z)\rightarrow
G_{\rho_ii,\rho_jj}(c,z)$, for
$(\rho_ii,\rho_jj),(\rho_kk,\rho_ll)\in \mathcal{J}$ and
$c_1,c_2,c\in \Gamma$. From the definition of $Y_{\rho_ii,\rho_jj}(c,z)$, $X_{\rho_ii,\rho_jj}(c,z)$, and note that the identity given in Lemma \ref{lem:4.3.}, we need to work on the product $B(z_1,z_2)\cdot C(z_1,z_2)$. For $C(z_1,z_2)$, we apply Proposition \ref{prop:3.6} to rewrite it into a summation of $\delta$ and $D(\delta)$ functions. For this purpose we divide the argument into the following seven cases.

Case $1$: $c_1\ne 1,c_2\ne 1,c_1c_2\ne 1$ and $c_1\ne c_2$;

Case $2$: $c_1=1$ and $c_2\ne 1$;

Case $3$: $c_2=1$ and $c_1\ne 1$;

Case $4$: $c_1=c_2=1$;

Case $5$: $c_1\ne 1,c_2\ne 1, c_1c_2=1$ and $c_1\ne c_2$;

Case $6$: $c_1\ne 1,c_2\ne 1, c_1=c_2$ and $c_1c_2\ne 1$;

Case $7$: $c_1\ne 1,c_2\ne 1, c_1c_2=1$ and $c_1=c_2$. \\

 For Case $1$. Recall that $\Gamma$ is generic, which implies that the numbers
$\omega^p,\omega^pc_1^{-1},$ $\omega^pc_2,
 \omega^pc_2c_1^{-1}, p\in \mathbb{Z}_m$ are distinct. Thus, by applying
Proposition \ref{prop:3.6}, we obtain

\begin{align*}
C(z_1,z_2) =&\sum_{r\in
\mathbb{Z}_m}\big[\delta_{\rho_ii,-\rho_kk_r}H_{1,r}
\delta(\omega^rz_2/z_1)
+\delta_{\rho_jj,\rho_kk_r}H_{2,r}\delta(\omega^rz_2/c_1z_1)\\
&+\delta_{\rho_ii,\rho_ll_r}H_{3,r} \delta(\omega^rc_2z_2/z_1)
+\delta_{\rho_jj,-\rho_ll_r}H_{4,r}
\delta(\omega^rc_2z_2/c_1z_1)\big],
\end{align*}
where
\begin{align*}
H_{1,r}=&\prod_{0<p<m}(1-\omega^p)^{\rho_i\rho_k\la\epsilon_i,\nu^{p+r}\epsilon_k\ra}
\prod_{p\in \mathbb{Z}_m}(1-\omega^p/c_1)^{-\rho_j\rho_k\la\epsilon_j,\nu^{p+r}\epsilon_k\ra}\\
&\cdot
(1-\omega^pc_2)^{-\rho_i\rho_l\la\epsilon_i,\nu^{p+r}\epsilon_l\ra}
(1-\omega^pc_2/c_1)^{\rho_j\rho_l\la\epsilon_j,\nu^{p+r}\epsilon_l\ra},
\end{align*}
\begin{align*}
H_{2,r}=&\prod_{0<p<m}(1-\omega^p)^{-\rho_j\rho_k\la\epsilon_j,\nu^{p+r}\epsilon_k\ra}
\prod_{p\in \mathbb{Z}_m}(1-\omega^pc_1)^{\rho_i\rho_k\la\epsilon_i,\nu^{p+r}\epsilon_k\ra}\\
&\cdot
(1-\omega^pc_1c_2)^{-\rho_i\rho_l\la\epsilon_i,\nu^{p+r}\epsilon_l\ra}
(1-\omega^pc_2)^{\rho_j\rho_l\la\epsilon_j,\nu^{p+r}\epsilon_l\ra},
\end{align*}
\begin{align*}
H_{3,r}=&\prod_{0<p<m}(1-\omega^p)^{-\rho_i\rho_l\la\epsilon_i,\nu^{p+r}\epsilon_l\ra}
\prod_{p\in \mathbb{Z}_m}(1-\omega^p/c_2)^{\rho_i\rho_k\la\epsilon_i,\nu^{p+r}\epsilon_k\ra}\\
&\cdot
(1-\omega^p/c_1c_2)^{-\rho_j\rho_k\la\epsilon_j,\nu^{p+r}\epsilon_k\ra}
(1-\omega^p/c_1)^{\rho_j\rho_l\la\epsilon_j,\nu^{p+r}\epsilon_l\ra},
\end{align*}
\begin{align*}
H_{4,r}=&\prod_{0<p<m}(1-\omega^p)^{\rho_j\rho_l\la\epsilon_j,\nu^{p+r}\epsilon_l\ra}
\prod_{p\in \mathbb{Z}_m}(1-\omega^pc_1/c_2)^{\rho_i\rho_k\la\epsilon_i,\nu^{p+r}\epsilon_k\ra}\\
&\cdot
(1-\omega^p/c_2)^{-\rho_j\rho_k\la\epsilon_j,\nu^{p+r}\epsilon_k\ra}
(1-\omega^pc_1)^{-\rho_i\rho_l\la\epsilon_j,\nu^{p+r}\epsilon_l\ra}.
\end{align*}

Therefore, by using identities (\ref{eq:4.1})and  (\ref{eq:4.3}-\ref{eq:4.7}), we have
\begin{align*} &[Y_{\rho_ii,\rho_jj}(c_1,z_1),Y_{\rho_kk,\rho_ll}(c_2,z_2)]\\
=&m^{-2}\sum_{r\in
\mathbb{Z}_m}\zeta(\underline{\alpha})\zeta(\underline{\beta})
\kappa(\rho_ii,\rho_jj,c_1)\kappa(\rho_kk,\rho_ll,c_2)
\big[\delta_{\rho_ii,-\rho_kk_r}H_{1,r} B(z_1,\omega^{-r}z_1)\\
&\cdot\delta(\omega^rz_2/z_1)
+\delta_{\rho_jj,\rho_kk_r}H_{2,r}B(z_1,\omega^{-r}c_1z_1)\delta(\omega^rz_2/c_1z_1)
+\delta_{\rho_ii,\rho_ll_r}H_{3,r}\\ &\cdot
B(z_1,\omega^{-r}c_2^{-1}z_1)\delta(\omega^rc_2z_2/z_1)
+\delta_{\rho_jj,-\rho_ll_r}H_{4,r}
B(z_1,\omega^{-r}c_1c_2^{-1}z_1)\delta(\omega^rc_2z_2/c_1z_1)\big]\\
=&m^{-1}\sum_{r\in \mathbb{Z}_m}
\kappa(\rho_ii,\rho_jj,c_1)\kappa(\rho_kk,\rho_ll,c_2)
\big(\prod_{0<p<m}(1-\omega^p)^{-\la\underline{\alpha},\sum\nu^p(\nu^r\underline{\beta})\ra}\big)\\
&\cdot\big[\delta_{\rho_ii,-\rho_kk_r}H_{1,r}c_1^{\la\rho_i\epsilon_i,\rho_j\sum\nu^p\epsilon_j\ra}
\kappa(-\rho_jj,\rho_ll_r,c_2c_1^{-1})^{-1}
\xi_r(\underline{\alpha},\underline{\beta})Y_{-\rho_jj,\rho_ll_r}(c_1^{-1}c_2,c_1z_1)
\\&\cdot\delta(\omega^rz_2/z_1)+\delta_{\rho_jj,\rho_kk_r}H_{2,r}\kappa(\rho_ii,\rho_ll_r,c_1c_2)^{-1}
\xi_r(\alpha,\beta)Y_{\rho_ii,\rho_ll_r}(c_1c_2,z_1)
\delta(\omega^rz_2/c_1z_1)\\
&+\delta_{\rho_ii,\rho_ll_r}H_{3,r}c_1^{\la\rho_i\epsilon_i,\rho_j\sum\nu^p\epsilon_j\ra}
c_2^{\la\rho_k\epsilon_k,\rho_l\sum\nu^p\epsilon_l\ra}
\kappa(-\rho_jj,-\rho_kk_r,c_1^{-1}c_2^{-1})^{-1}\\
&\cdot
\xi_r(\underline{\alpha},\underline{\beta})Y_{-\rho_jj,-\rho_kk_r}(c_1^{-1}c_2^{-1},c_1z_1)
\delta(\omega^rc_2z_2/z_1)+\delta_{-\rho_jj,\rho_ll_r}H_{4,r}c_2^{\la\rho_k\epsilon_k,\rho_l\sum\nu^p\epsilon_l\ra}\\
&
\cdot \kappa(\rho_ii,-\rho_kk_r,c_1c_2^{-1})^{-1}
\xi_r(\underline{\alpha},\underline{\beta})Y_{\rho_ii,-\rho_kk_r}(c_1c_2^{-1},z_1)
\delta(\omega^rc_2z_2/c_1z_1)\big].
\end{align*}

 Comparing the above commutator relation with (\ref{eq:cr}), we see that
 the result for case $1$ follows from the the following identities, which can be checked directly.

\begin{equation*}
\kappa(\rho_ii,\rho_jj,c_1)\kappa(\rho_kk,\rho_ll,c_2)
\prod_{0<p<m}(1-\omega^p)^{-\la\underline{\alpha},\nu^p(\nu^r\underline{\beta})\ra}
\end{equation*}
\begin{equation*}
=\begin{cases}
(-1)^{\delta_{ij}}c_1^{\la\rho_i\epsilon_j,-\rho_j\sum\nu^p\epsilon_i\ra}
\kappa(-\rho_jj,\rho_ll_r,c_2c_1^{-1})H_{1,r}^{-1},\  &\text{if } \rho_ii+\rho_kk_r=0,\\
\kappa(\rho_ii,\rho_ll_r,c_1c_2)H_{2,r}^{-1},\ &\text{if } \rho_jj=\rho_kk_r,\\
 (-1)^{\delta_{ij}+\delta_{kl}}c_1^{-\la\rho_j\epsilon_j,\rho_i\sum\nu^p\epsilon_i\ra}
 c_2^{-\la\rho_l\epsilon_l,\rho_k\sum\nu^p\epsilon_k\ra}& \\
\cdot \kappa(-\rho_jj,-\rho_kk_r,c_1^{-1}c_2^{-1})H_{3,r}^{-1},\ &\text{if } \rho_ii=\rho_ll_r,\\
(-1)^{\delta_{kl}}c_2^{-\la\rho_k\epsilon_k,\rho_l\sum\nu^p\epsilon_l\ra}
\kappa(\rho_ii,-\rho_kk_r,c_1c_2^{-1})H_{4,r}^{-1},\ &\text{if }
\rho_jj+\rho_ll_r=0.
\end{cases}\end{equation*}

 For Case $2$. We divide the proof of this case into two subcases.
First we consider the subcase for $\rho_ii=\rho_jj$. Since $c_1=1$, we
have $Y_{\rho_ii,\rho_jj}(1,z_1)=\rho_i\epsilon_i(z_1)$ and
$\underline{\alpha}=0$. The proof of this subcase is straightforward, and is omitted for shortness.

Next, we consider the other subcase for $\rho_ii\neq\rho_jj$.
If moreover $\rho_ii+\rho_jj=0$, by using (\ref{eq:R1}), we see that the result follows from the fact
$Y_{\rho_ii,\rho_jj}(1,z_1)=0$. On the other hand, if $\rho_ii+\rho_jj\not=0$, then $i\ne j$, and hence $|\la\underline{\alpha},\rho_k\nu^p\epsilon_k\ra|\le 1$ and
$|\la\underline{\alpha},\rho_l\nu^p\epsilon_l\ra|\le 1$. Together with (\ref{eq:3.15}) this gives
\begin{align*}
&C(z_1,z_2)
\\=&\sum_{r\in
\mathbb{Z}_m}(\delta_{\rho_ii,-\rho_kk_r}+\delta_{\rho_jj,\rho_kk_r})
L_{1,r}\delta(\omega^rz_2/z_1)+(\delta_{\rho_ii,\rho_ll_r}+\delta_{-\rho_jj,\rho_ll_r})
L_{2,r}\delta(\omega^rc_2z_2/z_1),
\end{align*}
where
\begin{align*}
L_{1,r}&=\prod_{0<p<m}(1-\omega^p)^{\la\underline{\alpha},\rho_k\nu^{p+r}\epsilon_k\ra}
 \prod_{p\in\mathbb{Z}_m}(1-\omega^pc_2)^{-\la\underline{\alpha},\rho_l\nu^{p+r}\epsilon_l\ra},\\
 L_{2,r}&=\prod_{0<p<m}(1-\omega^p)^{-\la\underline{\alpha},\rho_l\nu^{p+r}\epsilon_l\ra}
 \prod_{p\in \mathbb{Z}_m} (1-\omega^pc_2^{-1})^{\la\underline{\alpha},\rho_k\nu^{p+r}\epsilon_k\ra}.
 \end{align*}

Similar to the proof in Case $1$, by applying Lemmas \ref{lem:4.3},\ref{lem:4.4} and (\ref{eq:4.1}), we get
\begin{equation*}\begin{split}
&[Y_{\rho_ii,\rho_jj}(1,z_1),Y_{\rho_kk,\rho_ll}(c_2,z_2)]\\
&=m^{-1}\sum_{r\in \mathbb{Z}_m}\kappa(\rho_kk,\rho_ll,c_2)
\prod_{0<p<m}(1-\omega^p)^{-\la\alpha,\nu^p(\nu^r\beta)\ra}\bigg\{L_{1,r}\xi_r(\underline{\alpha},\underline{\beta})\delta(\omega^rz_2/z_1)\\
&\cdot\big[\delta_{\rho_ii,-\rho_kk_r}
\kappa(-\rho_jj,\rho_ll_r,c_2)^{-1} Y_{-\rho_jj,\rho_ll_r}(c_2,z_1)
+\delta_{\rho_jj,\rho_kk_r}\kappa(\rho_ii,\rho_ll_r,c_2)^{-1}
Y_{\rho_ii,\rho_ll_r}(c_2,z_1)\big]
\\
&+L_{2,r}c_2^{\la\rho_k\epsilon_k,\rho_l\sum\nu^p\epsilon_l\ra}
 \xi_r(\underline{\alpha},\underline{\beta})
\delta(\omega^rc_2z_2/z_1)\big[\delta_{\rho_ii,\rho_ll_r}
\kappa(-\rho_jj,-\rho_kk_r,c_2^{-1})^{-1}\\&
\cdot Y_{-\rho_jj,-\rho_kk_r}(c_2^{-1},z_1)
+\delta_{-\rho_jj,\rho_ll_r}
\kappa(\rho_ii,-\rho_kk_r,c_2^{-1})^{-1}
 Y_{\rho_ii,-\rho_kk_r}(c_2^{-1},z_1)\big)]\bigg\}.
\end{split}\end{equation*}

A direct computation shows the following identity.
\begin{equation*}\begin{split}
&\prod_{0<p<m}(1-\omega^p)^{-\la\alpha,\nu^p(\nu^r\beta)\ra}\kappa(\rho_kk,\rho_ll,c_2)\\
=&\begin{cases} \kappa(-\rho_jj,\rho_ll_r,c_2)L_{1,r}^{-1},\
&\text{if}\ \rho_ii+\rho_kk_r=0,\\
\kappa(\rho_ii,\rho_ll_r,c_2)L_{1,r}^{-1},\
&\text{if}\ \rho_jj=\rho_kk_r,\\
(-1)^{\delta_{kl}}c_2^{-\la\rho_k\epsilon_k,\rho_l\sum\nu^p\epsilon_l\ra}
\kappa(-\rho_jj,-\rho_kk_r,c_2^{-1})L_{2,r}^{-1},\ &\text{if}\ \rho_ii=\rho_ll_r,\\
(-1)^{\delta_{kl}}c_2^{-\la\rho_k\epsilon_k,\rho_l\sum\nu^p\epsilon_l\ra}
\kappa(\rho_ii,-\rho_kk_r,c_2^{-1})L_{2,r}^{-1},\ &\text{if}\
\rho_jj+\rho_ll_r=0.
\end{cases}\end{split}
\end{equation*}

The result for Case 2 then follows from the above two identities. It is clear that the result for Case 3 follows from Case 2 and Proposition \ref{prop:2.3}.

For Case $4$. We divide the argument into two subcases. The first one is for
$\rho_ii=\rho_jj$ or $\rho_kk=\rho_ll$, the other one is for
$\rho_ii\neq\rho_jj$ and $\rho_kk\neq\rho_ll$. In fact, if $\rho_ii=\rho_jj$
or $\rho_kk=\rho_ll$, then the proof is the same as that for Case 2.
Next we suppose $\rho_ii\neq\rho_jj$ and
$\rho_kk\neq\rho_ll$. If moreover $\rho_ii+\rho_jj=0$ or $\rho_kk+\rho_ll=0$,
then by (\ref{eq:R1}), we see that $Y_{\rho_ii,\rho_jj}(1,z_1)=0$
or $Y_{\rho_kk,\rho_ll}(1,z_2)=0$. Otherwise, $i\ne j$ and $k\ne l$, then
our vertex operator $Y_{\rho_ii,\rho_jj}(1,z)$ coincides with the one
defined in \cite{L}, and thus Theorem \ref{thm:4.1} follows from
Theorem 8.2 in \cite{L}.

 For Case 5.
In this case, the numbers
$\omega^p,\omega^pc_1^{-1},\omega^pc_1^{-1}c_2, p\in \mathbb{Z}_m$
are distinct. For $t=0,\pm 1,\pm 2$, we set $M(t)=\{r\in \mathbb{Z}_m|
-\la\rho_j\epsilon_j,\rho_k\nu^r\epsilon_k\ra
-\la\rho_i\epsilon_i,\rho_l\nu^p\epsilon_l\ra=t\}$. Then by
(\ref{eq:3.15}), we get
\begin{align*}&C(z_1,z_2)
=\sum_{r\in
\mathbb{Z}_m}[\delta_{\rho_ii,-\rho_kk_r}H_{1,r}\delta(\omega^rz_2/z_1)
+\delta_{\rho_jj,-\rho_ll_r}H_{4,r}\delta(\omega^rc_2z_2/c_1z_1)]+\\
&\sum_{r\in M(-1)}N_{1,r}
\delta(\omega^rz_2/c_1z_1) +\sum_{r\in
M(-2)}N_{1,r}[(D\delta)(\omega^rz_2/c_1z_1)+N_{2,r}\delta(\omega^rz_2/c_1z_1)],
\end{align*}
where
\begin{align*}
N_{1,r}=&
\prod_{0<p<m}(1-\omega^p)^{-\la\rho_j\epsilon_j,\rho_k\nu^{p+r}\epsilon_k\ra-\la\rho_i\epsilon_i,
\rho_l\nu^{p+r}\epsilon_l\ra}\\
&\cdot \prod_{p\in
\mathbb{Z}_m}(1-\omega^pc_1)^{\la\rho_i\epsilon_i,\rho_k\nu^{p+r}\epsilon_k\ra}
(1-\omega^p/c_1)^{\la\rho_j\epsilon_j,\rho_l\nu^{p+r}\epsilon_l\ra},\\
N_{2,r}=&1-\sum_{p\in
\mathbb{Z}_m}\left(\frac{\la\rho_i\epsilon_i,\rho_k\nu^p\epsilon_k\ra}
{1-\omega^{r-p}c_1^{-1}}+\frac{\la\rho_j\epsilon_j,\rho_l\nu^p\epsilon_l\ra}
{1-\omega^{r-p}c_1}\right)\\
&-\sum_{p\ne r\in \mathbb{Z}_m}
\frac{-\la\rho_j\epsilon_j,\rho_k\nu^p\epsilon_k\ra-\la\rho_i\epsilon_i,\rho_l\nu^p\epsilon_l\ra}
{1-\omega^{r-p}}
\end{align*}
and $H_{i,r},$ for $i=1,4$ are defined in the proof for Case $1$. This implies
$$
[Y_{\rho_ii,\rho_jj}(c_1,z_1),Y_{\rho_kk,\rho_ll}(c_2,z_2)]
=O_1+O_2+O_3+O_4,
$$
where
\begin{align*} &O_1=D(z_1,z_2)\sum_{r\in \mathbb{Z}_m}\delta_{\rho_ii,-\rho_kk_r}H_{1,r}\delta(\omega^rz_2/z_1),\\&
O_2=D(z_1,z_2)\sum_{r\in \mathbb Z_m} \delta_{\rho_jj,-\rho_ll_r}H_{4,r}\delta(\omega^rc_2z_2/c_1z_1),\\&
O_3=D(z_1,z_2)\sum_{r\in M(-1)}N_{1,r}\delta(\omega^rz_2/c_1z_1),\\
&O_4=D(z_1,z_2)\sum_{r\in
M(-2)}N_{1,r}[(D\delta)(\omega^rz_2/c_1z_1)+N_{2,r}\delta(\omega^rz_2/c_1z_1)],
\end{align*}
and
$
D(z_1,z_2)=m^{-2}\zeta(\underline{\alpha})\zeta(\underline{\beta})
\kappa(\rho_ii,\rho_jj,c_1)\kappa(\rho_kk,\rho_ll,c_2)B(z_1,z_2).
$

Similar to the proof given in Case $1$, we have
\begin{align}
O_1&=m^{-1}\sum_{r\in
\mathbb{Z}_m}\delta_{\rho_ii,-\rho_kk_r}(-1)^{\delta_{ij}}
\xi_r(\underline{\alpha},\underline{\beta})Y_{-\rho_jj,\rho_ll_r}(c_1^{-2},c_1z_1)
\delta(\omega^rz_2/z_1),\label{eq:O1}\\\
O_2&=m^{-1}\sum_{r\in
\mathbb{Z}_m}\delta_{-\rho_jj,\rho_ll_r}(-1)^{\delta_{kl}}
\xi_r(\underline{\alpha},\underline{\beta})Y_{\rho_ii,-\rho_kk_r}(c_1^2,z_1)
\delta(\omega^rz_2/c_1^2z_1).\label{eq:O2}
\end{align}

For the term $O_3$, note that $r\in M(-1)$, we know that either $\rho_ii_r=\rho_ll_r$ or
$\rho_jj=\rho_kk_r$. If $\rho_ii_r=\rho_ll_r$, then
$N_{1,r}=H_{3,r}$ as
$(1-\omega^p)^{-\la\rho_j\epsilon_j,\rho_k\nu^r\epsilon_k\ra}=1$.
Similarly, if $\rho_jj=\rho_kk_r$, then $N_{1,r}=H_{2,r}$.
Thus, by the same proof as that of Case $1$, we have
\begin{equation}\label{eq:O3}\begin{split}
&O_3=m^{-1}\sum_{r\in M(-1)}\delta_{\rho_jj,\rho_kk_r}
\xi_r(\underline{\alpha},\underline{\beta})Y_{\rho_ii,\rho_ll_r}(1,z_1)
\delta(\omega^rz_2/c_1z_1)\\
&+m^{-1}\sum_{r\in
M(-1)}\delta_{\rho_ii,\rho_ll_r}(-1)^{\delta_{ij}+\delta_{kl}}
\xi_r(\underline{\alpha},\underline{\beta})Y_{-\rho_jj,-,\rho_kk_r}(1,c_1z_1)
\delta(\omega^rz_2/c_1z_1).
\end{split}\end{equation}

Now we consider the term $O_4$. First, by definition, one can see that
 \begin{equation}\label{eq:4.10}N_{2,r}=\frac{1}{2}\la\underline{\beta},\sum\nu^p\underline{\beta}\ra, \text{ if }
 r\in M(-2).
 \end{equation}
  By using (\ref{eq:4.2}) and (\ref{eq:4.4}), we have
\begin{align*} &B(z_1,z_2)(D\delta)(\omega^rz_2/c_1z_1)\\
=&B(z_1,\omega^{-r}c_1z_1)(D\delta)(\omega^rz_2/c_1z_1)
-m(\rho_k\epsilon_k(z_2)-\rho_l\epsilon_l(c_2z_2)\\&+\la\underline{\beta},\sum\nu^p\underline{\beta}\ra/2)
B(z_1,\omega^{-r}c_1z_1)\delta(\omega^rz_2/c_1z_1).
\end{align*}

Together with \eqref{eq:4.10} this gives
 \begin{equation}\label{eq:4.11}\begin{split}
&B(z_1,z_2)[(D\delta)(\omega^rz_2/c_1z_1)+N_{2,r}
\delta(\omega^rz_2/c_1z_1)]\\
=&\varepsilon_C(\underline{\alpha},\nu^r\underline{\beta})\eta(r,\underline{\beta})
[(D\delta)(\omega^rz_2/c_1z_1)+(\rho_i\epsilon_i(z_1)-\rho_i\epsilon_i(c_1z_1))
\delta(\omega^rz_2/c_1z_1)].
\end{split}\end{equation}

It follows from (\ref{eq:4.7}) that
\begin{equation*}
\zeta(\underline{\alpha})\zeta(\underline{\beta})\kappa(\rho_ii,\rho_jj,c_1)
\kappa(\rho_kk,\rho_ll,c_1^{-1})N_{1,r}=\varepsilon'(\underline{\alpha},\nu^r\underline{\beta}),
 \text{ if }r\in M(-2).
\end{equation*}
Together with (\ref{eq:4.11}) this implies
\begin{equation}\label{eq:O4}\begin{split}
O_4=&m^{-1}\sum_{r\in
M(-2)}\xi_r(\underline{\alpha},\underline{\beta})(\rho_i\epsilon_i(z_1)-\rho_j\epsilon_j(c_1z_1))
\delta(\omega^rz_2/c_1z_1)\\
&+m^{-2}\sum_{r\in M(-2)}\xi_r(\underline{\alpha},\underline{\beta})
(D\delta)(\omega^rz_2/c_1z_1).
\end{split}\end{equation}

Combining (\ref{eq:O1}), (\ref{eq:O2}), (\ref{eq:O3}) and (\ref{eq:O4}),
and comparing the expression of $O_1+O_2+O_3+O_4$ with
(\ref{eq:cr}), one sees that the result for the first part of the theorem under the condition of Case 5 follows from the following identity.
\begin{align*}
&O_3+m^{-1}\sum_{r\in
M(-2)}\xi_r(\underline{\alpha},\underline{\beta})(\rho_i\epsilon_i(z_1)-\rho_j\epsilon_j(c_1z_1))
\delta(\omega^rz_2/c_1z_1)\\
=&m^{-1}\sum_{r\in \mathbb{Z}_m}\delta_{\rho_jj,\rho_kk_r}
\xi_r(\underline{\alpha},\underline{\beta})Y_{\rho_ii,\rho_ll_r}(1,z_1)
\delta(\omega^rz_2/c_1z_1)\\
&+m^{-1}\sum_{r\in
\mathbb{Z}_m}\delta_{\rho_ii,\rho_ll_r}(-1)^{\delta_{ij}+\delta_{kl}}
\xi_r(\underline{\alpha},\underline{\beta})Y_{-\rho_jj,-\rho_kk_r}(1,c_1z_1)
\delta(\omega^rz_2/c_1z_1),
\end{align*}

It is also clear that the result for Case 6 follows from Case 5 and Proposition \ref{prop:2.3}. Finally, for Case 7, we note that in this case one has $c_1=c_2=-1$, but this is impossible as the group $\Gamma$ is generic. Therefore, we have finished the proof of the first part of the theorem.
The proof of the second part of the theorem is standard, which is omitted. \qed
\section{Applications}

This section is devoted to the application of Theorem \ref{thm:4.1}. By choosing some special quadruples
$(Q,\nu,m,\Gamma)$, we recover vertex operator representations  presented in \cite{L,G1,G2,BS,BGT,G-KL1,G-KL2,CGJT,CT}.
We also provide a vertex operator
representation for the $BC_{N-1}$-graded Lie algebra
$\widehat{\mathfrak o}_{2N}^{(2)}(\mathbb C_\Gamma)$.

\subsection{Realization of twisted affine Lie algebras.}
Following \cite{L}, we define a Lie algebra $\mathfrak g=\mathcal
H\oplus \sum_{\alpha\in Q'} \mathbb Cx_\alpha$, where $\mathcal
H=P\otimes_\mathbb Z\mathbb C$ and $Q'=\{\alpha\in Q|\la
\alpha,\alpha\ra=2\}$, with Lie bracket
\begin{equation*}\begin{split}
[\mathcal H,\mathcal H]=0,\ [h,x_\alpha]=\la h,\alpha\ra x_\alpha=-[x_\alpha,h],\\
[x_\alpha,x_\beta]=\begin{cases}
\varepsilon(\alpha,-\alpha)\alpha,\ &\text{if}\ \alpha+\beta=0,\\
\varepsilon(\alpha,\beta)x_{\alpha+\beta},\ &\text{if}\ \la\alpha,\beta\ra=-1,\\
0,\ &\text{if}\ \la\alpha,\beta\ra\ge 0,
\end{cases}\end{split}\end{equation*}
for $h\in \mathcal H, \alpha,\beta\in Q'$.
Extend the bilinear form $\la,\ra$ of Cartan
subalgebra $\mathcal H$ of $\mathfrak g$ to be an invariant form on
$\mathfrak g$ by
 $$\la h,x_\alpha\ra=0, \la x_\alpha,x_\beta\ra=\delta_{\alpha+\beta,0}\varepsilon(\alpha,-\alpha),
 \qquad h\in \mathcal H,\alpha,\beta\in Q'.$$
 And extend the linear automorphism $\nu$ of $\mathcal H$ to be a Lie automorphism of
$\mathfrak g$ by
 $$\nu(x_\alpha)=\eta(1,\alpha)x_{\nu (\alpha)},\qquad \alpha\in Q'.$$
 For $n\in \mathbb Z$ and $x\in \mathfrak g$, we
set
$x_{(n)}=m^{-1}\sum_{p\in\mathbb Z_m}\omega^{-np}\nu^p(x)\ \text{and}\ \mathfrak g_{(n)}=\{x_{(n)}|x\in \mathfrak g\}.$
Consider the twisted affine Lie algebra $\widehat{\mathfrak g}(\nu)=\sum_{n\in \mathbb Z} \mathfrak g_{(n)}\otimes t^n \oplus \mathbb C \mathbf c,$
with Lie bracket
\begin{equation*}
[x\otimes t^n,y\otimes t^r]=[x,y]\otimes t^{n+r}+m^{-1}n\la
x,y\ra\delta_{n+r,0}\mathbf c
\end{equation*}
for $x\in \mathfrak g_{(n)}, y\in \mathfrak g_{(r)}$, $n,r\in
\mathbb Z$, and $\mathbf c$ is central.

Take $\Gamma=\{1\}$. One can check that the Lie algebra $\widehat{\mathcal G}(Q,\nu,m,\{1\})$ is
isomorphic to $\widehat{\mathfrak g}(\nu)$ via the isomorphism
 \begin{align*}
\widetilde{e_{i,\rho_jj}}(1,n)\mapsto
(x_{\epsilon_i-\rho_j\epsilon_j})_{(n)}\otimes t^n,\
\widetilde{e_{k,k}}(1,n)\mapsto (\epsilon_k)_{(n)}\otimes t^n,\ \mathbf c\mapsto \mathbf c,
\end{align*}
for $(i,\rho_jj)\in \mathcal J, i\ne j, 1\le k\le N$ and $n\in \mathbb Z$.
Comparing this with Theorem \ref{thm:4.1}, we obtain the following result
which was given in \cite{L}.
\begin{cor}\label{cor:5.1} The generalized Fock space $V_T$ affords a
representation of the $\nu$-twisted affine Lie algebra
$\widehat{\mathfrak g}(\nu)$ with action given by
$$(x_{\epsilon_i-\rho_j\epsilon_j})_{(n)}\otimes t^n\mapsto y_{i,\rho_jj}(1,n),\ (\varepsilon_k)_{(n)}\otimes t^n \mapsto
y_{k,k}(1, n),\ \mathbf c\mapsto 1,$$ for $(i,\rho_jj)\in \mathcal
J, i\ne j, 1\le k\le N$ and $n\in \mathbb Z$. In particular, if $\mathrm{span}_\mathbb ZQ_{(2)}=Q$,
then  $V_T$ is irreducible if and only if the $\mathbb C[Q,\varepsilon_C]$-module $T$ is
irreducible. \qed
\end{cor}

\subsection{Realization of extended affine Lie algebras of type $A_{N-1}$}
In this section, we  present the homogenous and principal vertex
operator representations of the Lie algebra
$\widehat{\mathfrak{gl}}_N(\mathbb C_q)$. Let $Q=Q(A_{N-1}), N\geq
2, \nu=\mathrm{Id}, m=1$ and $\Gamma=\Gamma_q$ (cf. Section 2.2).
Note that in this case we have
$C(\alpha,\beta)=(-1)^{\la\alpha,\beta\ra}$, thus there is a
$2$-cocycle $\varepsilon^*: P\times P\to \{\pm 1\}$ associated with
$C$ determined  by $$
 \varepsilon^*(\sum m_i\epsilon_i, \sum
 n_j\epsilon_j)=\prod_{i,j}(\varepsilon^*(\epsilon_i,\epsilon_j))^{m_in_j},
 $$
where
 $\varepsilon^*(\epsilon_i,\epsilon_j)=
 1\ \text{if}\ i\leq j\ \text{and}\ =
 -1\ \text{if}\ i>j.$

 Define a $\mathbb C[Q,\varepsilon^*]$-module structure
  and  $\mathcal H_{(0)}$-action on the group algebra
 $\mathbb C[Q]=\oplus_{\alpha\in Q}\mathbb Ce^\alpha$ as follows
\begin{equation}\label{eq:2.12}\begin{split}
e_\alpha. e^\beta=\varepsilon^*(\alpha,\beta)e^{\alpha+\beta},\quad
h.e^\alpha=\la h,\alpha\ra e^\alpha,\ h\in \mathcal H_{(0)}=\mathcal
H,\alpha,\beta\in Q,
\end{split}\end{equation}
which is obviously compatible in the sense of (\ref{eq:2.2}). Thus
we may take $\varepsilon_C=\varepsilon^*$ and let $V_T=\mathbb
C[Q]\otimes S$. One can check that the Lie algebra $\widehat{\mathcal G}(Q(A_{N-1}),\mathrm{Id},1,\Gamma_q)$ is
isomorphic to $\widehat{\mathfrak{gl}}_N(\mathbb C_q)$ via the isomorphism
 \begin{align*}
\varepsilon^*(\epsilon_i,\epsilon_j)\widetilde{e_{i,j}}(q^{\mathbf n},n_0)\mapsto
E_{i,j}t_0^{n_0}t^\mathbf n,\ \mathbf c\mapsto \mathbf c,\ 1\le i,j\le N, (n_0,\mathbf n)\in \mathbb Z^{l+1}.
\end{align*}
Comparing this with Theorem \ref{thm:4.1}, we obtain the following result,
which was given in \cite{G1,BGT}.
\begin{cor} There is an irreducible representation of the Lie
algebra $\widehat{\mathfrak{gl}}_N(\mathbb C_q)$ on $V_{T}=\mathbb
C[Q(A_{N-1})]\otimes S$ in the homogeneous picture. The
representation is given by the mapping
\begin{align*} E_{i,j}t_0^{n_0}t^\mathbf n\mapsto
\varepsilon^*(\epsilon_i,\epsilon_j)y_{i,j}(q^{\mathbf n},n_0),\
\mathbf c\mapsto 1,\ 1\le i,j\le N, n\in \mathbb Z, \mathbf n\in \mathbb Z^l. \qed
\end{align*}
\end{cor}

In what follows we present the principal realization of
$\widehat{\mathfrak{gl}}_N(\mathbb C_{q^N})$, where
$q^N=(q_1^N,\cdots,q_l^N)$. Set
$E=E_{12}+\cdots +E_{N-1,N}+E_{N,1}\text{ and } F=\sum_{i=1}^N\omega^iE_{ii},$
where $\omega$ is a primitive $N$-th  root of unity. It was proved in \cite{G2} that the
subalgebra of $\widehat{\mathfrak{gl}}_N(\mathbb C_q)$ spanned by
the elements
$F^iE^{n_0}(t_0^{n_0}t^\mathbf n), \mathbf c,\ i,n_0\in \mathbb Z, \mathbf n\in \mathbb Z^l, $
is isomorphic to the Lie algebra $\widehat{\mathfrak{gl}}_N(\mathbb
C_{q^N})$. For $i,j,n_0,r_0\in \mathbb Z, \mathbf n,\mathbf r\in \mathbb Z^l$, one has that
\begin{equation}\label{eq:7.2} \begin{split}
&[F^iE^{n_0}(t_0^{n_0}t^\mathbf n),F^jE^{r_0}(t_0^{r_0}t^\mathbf r)]
=\omega^{jn_0}q^{r_0\mathbf n}F^{i+j}E^{n_0+r_0}(t_0^{n_0+r_0}t^{\mathbf n+\mathbf r})\\
&-\omega^{ir_0}q^{n_0\mathbf
r}F^{i+j}E^{n_0+r_0}(t_0^{n_0+r_0}t^{\mathbf n+\mathbf r})
+n_0q^{r_0\mathbf
n}\omega^{jn_0}\delta_{\overline{i+j},N}\delta_{n_0+r_0,0}\delta_{\mathbf
n+\mathbf r,0}\mathbf c,
\end{split}
\end{equation}
where $\bar{i}$ is the unique integer in $\{1,\cdots,N\}$ such that
$\bar{i}\equiv i(\mathrm{mod}\ N)$.

Choose $Q=Q(A_{N-1}), N\geq 2, \nu=\nu_c,m=N$ and $\Gamma=\Gamma_q$,
where $\nu_c$ is the isometry of $P$ defined by
$\nu_c(\epsilon_i)=\epsilon_{\sigma(i)},\ 1\le i\le N,\sigma=(12\cdots N).$
Thus $\nu_c$ is the Coxeter isometry of $Q(A_{N-1})$ and
satisfies the following conditions
\begin{align*}
&\sum_{p\in \mathbb Z_m}\nu_c^p\alpha=0,\ \forall \alpha\in Q,\ \sum_{p\in \mathbb Z_m}\la p\nu_c^p\alpha,\beta\ra\in m\mathbb Z,\
\forall \alpha,\beta\in Q.
\end{align*}
This implies $C(\alpha,\beta)=1$ for all $\alpha,\beta\in Q$. Thus
we may take $\varepsilon_C=1$ and $\eta(p,\alpha)=1$ for all $p\in
\mathbb Z_m,\alpha\in Q$. Then the trivial
$\mathbb C[Q,\varepsilon_C]$-module $\mathbb C$ satisfies the
condition (\ref{eq:2.2}) and hence we may take the generalized Fock
space to be $S$.

It follows from
(\ref{eq:4.7}) and Proposition
\ref{prop:cr} that, for $i,j\in \mathbb Z$ and $\mathbf n,\mathbf r\in \mathbb Z^l$,
\begin{equation}\begin{split}\label{eq:6.3}
&[G^{i}(\mathbf n,z_1),G^{j}(\mathbf r,z_2)]
=G^{i+j}(\mathbf n+\mathbf r,\omega^{-j} z_1)\delta(\omega^jz_2/q^\mathbf nz_1)\\
&-G^{i+j}(\mathbf n+\mathbf r,\omega^{-i}
z_2)\delta(\omega^iz_1/q^\mathbf
rz_2)+\delta_{\overline{i+j},N}(D\delta)(\omega^jz_2/q^\mathbf nz_1)
\end{split}\end{equation}
where
$G^{i}(\mathbf n,z):=N\zeta(\epsilon_{\overline{N-i+1}}-\epsilon_{1})^{-1}(1-\omega^{-i})^{\delta_{\bar{i},N}-1}
G_{\overline{N-i+1},1}(q^{\mathbf{n}},z).$
Let $G^{i}(\mathbf n,z)=\sum_{n_0\in \mathbb Z}g^i(\mathbf n,n_0)z^{-n_0}$, then
by comparing the identity (\ref{eq:6.3}) with (\ref{eq:7.2}), we obtain the following isomorphism between
the Lie algebra $\widehat{\mathfrak{gl}}_N(\mathbb C_{q^N})$ and $\widehat{\mathcal G}(Q(A_{N-1}),\nu_c,N,\Gamma_q)$
\begin{align*}
F^iE^{n_0}t_0^{n_0}t^\mathbf n\mapsto g^i(\mathbf n,n_0),\ \mathbf
c\mapsto \mathbf c,\ i, n_0\in \mathbb Z, \mathbf n\in \mathbb Z^l.
\end{align*}

Set $y^i(\mathbf n,n_0):=
N\zeta(\epsilon_{\overline{N-i+1}}-\epsilon_{1})^{-1}(1-\omega^{-i})^{\delta_{\bar{i},N}-1}
y_{\overline{N-i+1},1}(q^{\mathbf{n}},n_0)$ for $i, n_0\in \mathbb
Z$ and $\mathbf n\in \mathbb Z^l$. Thus from Theorem \ref{thm:4.1}, we obtain the following result which was
given in \cite{BS} for the case $N=2$ and in \cite{G2,BGT} for all $N\ge 2$.
\begin{cor} There is an irreducible representation of the Lie
algebra $\widehat{\mathfrak{gl}}_N(\mathbb C_{q^N})$ on $V_T=S$ in
the principal picture. The representation is given by the mapping
\begin{align*}
F^iE^{n_0}t_0^{n_0}t^\mathbf n\mapsto y^i(\mathbf n,n_0),\ \mathbf
c\mapsto 1,\ i, n_0\in \mathbb Z,\mathbf n\in \mathbb Z^l. \qed
\end{align*}
\end{cor}

\subsection{Realization of trigonometric Lie algebras}
 Let $N=1,Q=\{0\},\nu=\mathrm{Id}$ or $-\mathrm{Id},m=1$
or $2$ and $\Gamma=\Gamma_{\mathbf h}$ (cf.Section 2.2). Note that in this case the
generalized Fock space is  $V_T=S(\mathcal H(\nu)^-)$. Set $g(\mathbf
n,n_0)=e^{-n_0\sqrt{-1}(\mathbf h,\mathbf
n)}\widetilde{e_{1,1}}(e^{-2\sqrt{-1}(\mathbf h,\mathbf n)},n_0)$
for $(n_0,\mathbf n)\in \mathbb Z^{l+1}$. From the third example in
Section 2.2 and Proposition \ref{prop:cr}, one can check that the
trigonometric Lie algebra of series $\widehat{A}_\mathbf h$ (resp. $\widehat{B}_\mathbf h$) is
isomorphic to $\widehat{\mathcal
G}(\{0\},\mathrm{Id},1,\Gamma_\mathbf h)$ (resp. $\widehat{\mathcal
G}(\{0\},-\mathrm{Id},2,\Gamma_\mathbf h)$) via the isomorphism
$$A_{\mathbf n,n_0}\ (\text{resp}.\ 2B_{\mathbf n,n_0}\ ) \mapsto g(\mathbf n,n_0),\ \mathbf c\mapsto \mathbf c,\ (n_0,\mathbf n)\in \mathbb Z^{l+1}.$$

Set $y(\mathbf n,n_0)=e^{-n_0\sqrt{-1}(\mathbf h,\mathbf
n)}y_{1,1}(e^{-2\sqrt{-1}(\mathbf h,\mathbf n)},n_0)$ for
$(n_0,\mathbf n)\in \mathbb Z^{l+1}$. From the isomorphisms given
above and Theorem \ref{thm:4.1}, we have the
following result, which was given in \cite{G-KL1,G-KL2}.
\begin{cor} The generalized Fock spaces $S(\mathcal
H(\mathrm{Id})^-)$ and $S(\mathcal H(-\mathrm{Id})^-)$ afford irreducible representations for the trigonometric
Lie algebras of
series $\widehat{A}_\mathbf h$ and $\widehat{B}_\mathbf h$ respectively.
The representations are respectively given by the mapping
\begin{align*}
A_{\mathbf n,n_0}\mapsto y(\mathbf n,n_0),\ \mathbf c\mapsto 1,
\quad \text{and}\quad
B_{\mathbf n,n_0}\mapsto 2y(\mathbf n,n_0),\ \mathbf c\mapsto 1, \ (n_0,\mathbf n)\in \mathbb Z^{l+1}. \qed
\end{align*}
\end{cor}

\subsection{Realization of unitary Lie algebras}
Choose $Q=Q(A_{N-1}),\nu=-\mathrm{Id},m=2$. Then
$C(\alpha,\beta)=(-1)^{\la \alpha,\beta \ra}$ for $\alpha,\beta\in
Q$, and we may take $\varepsilon_C=\varepsilon^*$, and
$\eta(p,\alpha)=1$ for all $\alpha\in Q,p\in \mathbb Z_m$ as $\nu$
preserves $\varepsilon^*$. Note that $\mathcal H_{(0)}=0$ in this
case and hence the condition (\ref{eq:2.2}) is equivalent to that of
$e_{2\alpha}$ acting as the identity operator on $T$ for any
$\alpha\in Q$. Let $\mathbb C[P/2P]=\oplus_{\alpha\in P}\mathbb
Ce_{\bar{\alpha}}, \bar{\alpha}= \alpha+2P$ be the group algebra
over the quotient group $P/2P$. Define a $\mathbb
C[Q,\varepsilon^*]$-module structure on $\mathbb C[P/2P]$ by
 \begin{align*}
 e_\alpha.e_{\bar{\beta}}=\varepsilon^*(\alpha,\beta)e_{\overline{\alpha+\beta}},\ \alpha\in Q,\beta\in P.
 \end{align*}
 Obviously, $e_{2\alpha}$ acts as the identity operator on $\mathbb C[P/2P]$.
 Therefore, we take $V_T=\mathbb C[P/2P]\otimes S$.
 One can check that the Lie algebra $\widehat{\mathcal G}(Q(A_{N-1}),-\mathrm{Id},2,\Gamma)$ is
isomorphic to the unitary Lie algebra
$\widehat{\mathfrak{u}}_N(\mathbb C_\Gamma)$ via the isomorphism
 \begin{align*}
2\varepsilon^*(\epsilon_i,\epsilon_j)\widetilde{e_{i,j}}(c,n)\mapsto
u_{i,j}(c,n), \mathbf c\mapsto \mathbf c,\ 1\le i,j\le N, c\in \Gamma, n\in \mathbb Z.
\end{align*}
Comparing this with Theorem \ref{thm:4.1}, we have the following result obtained in \cite{CGJT}.

\begin{cor} The  generalized Fock space $V_T=\mathbb
C[P/2P]\otimes S$ affords a representation for the unitary Lie algebra
$\widehat{\mathfrak{u}}_N(\mathbb C_\Gamma)$ with the actions given
by
\begin{align*}
u_{i,j}(c,n)\mapsto
2\varepsilon^*(\epsilon_i,\epsilon_j)y_{i,j}(c,n),\ \mathbf c\mapsto
1,\ 1\le i,j\le N, n\in \mathbb Z, c\in \Gamma. \qed
\end{align*}
\end{cor}

\subsection{Realization of the $BC_N$-graded Lie algebra
$\widehat{\mathfrak{o}}_{2N}(\mathbb C_\Gamma)$} Choose $Q=Q(D_N),$ $
\nu=\mathrm{Id}$ and $m=1$. By a similar argument as we did in Section 6.2, we
may take
$\varepsilon_C=\varepsilon^*$ and the $\mathbb
C[Q,\varepsilon^*]$-module $T$ to be $\mathbb C[Q(D_N)]$ with action defined in
(\ref{eq:2.12}).  Note that the Lie algebra $\widehat{\mathcal
G}(Q(D_N),\mathrm{Id},1,\Gamma)$ is isomorphic to the Lie algebra
$\widehat{\mathfrak{o}}_{2N}(\mathbb C_\Gamma)$ with the mapping
given by
$$\varepsilon^*(\epsilon_i,\epsilon_j)\widetilde{e_{\rho_ii,\rho_jj}}(c,n)\mapsto
f_{\rho_ii,\rho_jj}(c,n),\ \mathbf c\mapsto \mathbf c,\
1\le i,j\le N, \rho_i,\rho_j=\pm 1, n\in \mathbb Z,c\in
\Gamma.
$$
Thus, by Theorem
\ref{thm:4.1}, we have the following result, which was given in
\cite{CT}.
\begin{cor} There is an irreducible $\widehat{\mathfrak{o}}_{2N}(\mathbb C_\Gamma)$-module structure
on the generalized  Fock space $V_T=\mathbb C[Q(D_N)]\otimes S$ with
action given by
\begin{align*}
f_{\rho_ii,\rho_jj}(c,n)\mapsto
\varepsilon^*(\epsilon_i,\epsilon_j)y_{\rho_ii,\rho_jj}(c,n),\
\mathbf c_\mapsto 1,
\end{align*}
for $1\le i,j\le N, \rho_i,\rho_j=\pm 1, n\in \mathbb Z$ and $c\in
\Gamma$. \qed
\end{cor}

\subsection{Realization of the $BC_{N-1}$-graded Lie algebra $\widehat{\mathfrak o}_{2N}^{(2)}(\mathbb C_\Gamma)$}
In this section we give a homogeneous vertex
operator construction for the $BC_{N-1}$-graded Lie algebras
$\widehat{\mathfrak o}_{2N}^{(2)}(\mathbb C_\Gamma)$
 with grading subalgebra of type $B_{N-1}$ defined in \cite{ABG}.
In what follows we take $Q=Q(D_N), \nu=\nu_d$ and $m=2$, where
$\nu_d$ is the diagram automorphism of $Q(D_N)$. Recall that
$\nu_d(\epsilon_i)=\epsilon_i$ for $i=1,\cdots,N-1$ and
$\nu_d(\epsilon_N)=-\epsilon_N$. Then we have $i_1=i$
for $i=1,\cdots,N-1$ and $N_1=-N$.

Define an involution $^*$ of the Lie algebra $\widehat{\mathfrak
o}_{2N}(\mathbb C_\Gamma)$ (cf. Section 2.2) as follows
 $$f_{\rho_ii,\rho_jj}(c,n)^*= (-1)^nf_{\rho_ii_1,\rho_jj_1}(c,n),\ \mathbf c\mapsto \mathbf c,$$
where $1\le i,j\le N,\rho_i,\rho_j=\pm1, c\in \Gamma$ and $n\in
\mathbb Z$. We denote by $\widehat{\mathfrak o}_{2N}^{(2)}(\mathbb
C_\Gamma)$ the subalgebra consisting of fixed-points of $\widehat{\mathfrak
o}_{2N}(\mathbb C_\Gamma)$ under the involution $^*$.
One can easily check that the Lie algebra $\widehat{\mathfrak
o}_{2N}^{(2)}(\mathbb C_\Gamma)$ is a $BC_{N-1}$-graded Lie algebra
 with grading subalgebra of type $B_{N-1}$ in the sense of \cite{ABG}.
 We remark that if
$\Gamma=\{1\}$, then this fixed-point subalgebra is nothing but the affine Lie algebra $\widehat{\mathfrak o}_{2N}^{(2)}(\mathbb C)$.

Note that in the case $\nu=\nu_d$ and $m=2$, we have
$C(\alpha,\beta)=(-1)^{\la \alpha,\beta\ra}$ for $\alpha,\beta\in
Q(D_N)$. Thus we may choose $\varepsilon_C=\varepsilon^*$ and
$\eta(p,\alpha)=1$ for $p\in \mathbb Z_m$ and $\alpha\in Q(D_N)$.
 Let $\mathbb C[P/2\mathbb
Z\epsilon_N]=\oplus_{\alpha\in P}\mathbb Ce_{\bar{\alpha}}$,
$\bar{\alpha}= \alpha+2\mathbb Z\epsilon_N$ be the group algebra
over the quotient group $P/2\mathbb Z\epsilon_N$. Define a $\mathbb
C[Q(D_{N}),\varepsilon^*]$-module structure and an $\mathcal
H_{(0)}$-action
 on $\mathbb C[P/2\mathbb Z\epsilon_N]$ by
 \begin{align*}
 e_\alpha.e_{\bar{\beta}}&=\varepsilon^*(\alpha,\beta)e_{\overline{\alpha+\beta}},\
 h.e_{\bar{\beta}}=\la h,\beta \ra e_{\bar{\beta}},\ \alpha\in Q(D_N),\beta\in P, h\in \mathcal H_{(0)},
 \end{align*}
which is compatible in the sense of (\ref{eq:2.2}). Therefore, we may
take the generalized Fock space to be $\mathbb C[P/2\mathbb
Z\epsilon_N]\otimes S$. For $j=0,1$, set
\begin{align*} \mathbb C[P/2\mathbb Z\epsilon_N]^j=\{e_{\bar{\alpha}}|\bar{\alpha}=\sum_{i=1}^{N} a_i\bar{\epsilon}_i,a_1,\cdots,a_{N-1}\in \mathbb
Z, a_N=0,1,
\sum_{i=1}^N a_i\in 2\mathbb Z+j\},
\end{align*}
which are
irreducible $\mathbb C[Q(D_N),\varepsilon^*]$-submodules of $\mathbb
C[P/2\mathbb Z\epsilon_N]$.

One can check that
the Lie algebra $\widehat{\mathfrak o}_{2N}^{(2)}(\mathbb C_\Gamma)$
 is isomorphic to $\widehat{\mathcal G}(Q(D_N),\nu_d,2,\Gamma)$
with the isomorphism given by
$$g_{\rho_ii,\rho_jj}(c,n)
\mapsto
\varepsilon^*(\epsilon_i,\epsilon_j)\widetilde{e_{\rho_ii,\rho_jj}}(c,n),\
2\mathbf c\mapsto \mathbf c,\ 1\le i,j\le N,\rho_i,\rho_j=\pm
1, c\in \Gamma,n\in \mathbb Z.$$ This together with Theorem
\ref{thm:4.1} gives us the following result.
\begin{thm} \label{thm:6.5} There is an $\widehat{\mathfrak o}_{2N}^{(2)}(\mathbb C_\Gamma)$-module structure
on $\mathbb C[P/2\mathbb Z\epsilon_N]\otimes S$ by the mapping
\begin{align*}
g_{\rho_ii,\rho_jj}(c,n)\mapsto
\varepsilon^*(\epsilon_i,\epsilon_j)y_{\rho_ii,\rho_jj}(c,n),\
2\mathbf c_\mapsto 1,
\end{align*}
for $1\le i,j\le N, \rho_i,\rho_j=\pm 1, n\in \mathbb Z$ and $c\in
\Gamma$. Moreover, the $\widehat{\mathfrak o}_{2N}^{(2)}(\mathbb
C_\Gamma)$-module $\mathbb C[P/2\mathbb Z\epsilon_N]\otimes S$ is
completely reducible and the irreducible components are $\mathbb
C[P/2\mathbb Z\epsilon_N]^0\otimes S$ and $\mathbb C[P/2\mathbb
Z\epsilon_N]^1\otimes S$. \qed
\end{thm}

\end{document}